\documentclass[a4paper, 12pt]{amsart}
\usepackage{amsmath, amsthm, amssymb, fullpage}
\usepackage{graphicx}
\usepackage{enumerate}
\usepackage{xspace}
\usepackage{bbm} 

\usepackage[colorlinks]{hyperref}

\newtheorem{thm}{Theorem}[section]
\newtheorem{lem}[thm]{Lemma}

\newtheorem{rmk}[thm]{Remark}
\newtheorem{cor}[thm]{Corollary}
\newtheorem{defn}[thm]{Definition}

\newtheorem{ex}[thm]{Example}

\newcommand{\N}{\mathbb{N}}
\newcommand{\R}{\mathbb{R}}

\newcommand{\mcl}{\mathcal L}

\newcommand{\mcz}{\mathcal Z}

\newcommand{\al}{\alpha}
\newcommand{\be}{\beta}

\newcommand{\del}{\delta}
\newcommand{\ep}{\epsilon}
\newcommand{\sig}{\sigma}
\newcommand{\ka}{\kappa}
\newcommand{\lam}{\lambda}

\renewcommand{\t}[1]{\tilde{#1}} 

\newcommand{\mc}[1]{\mathcal{#1}}

\newcommand{\leb}{\text{Leb}}

\newcommand{\essinf}{\mathop{\mathrm{essinf}}}
\newcommand{\var}{\mathop{\mathrm{var}}}
\newcommand{\spt}{\mathop{\mathrm{supp}}}
\newcommand{\supp}{\mathop{\mathrm{supp}}}
\newcommand\Lk{\mathcal{L}_k}

\def\esssup{\mbox{ess\,sup}}

\title{Ulam's method for Lasota-Yorke maps with holes}
\author[C. Bose, G. Froyland, C. Gonz\'alez-Tokman and R. Murray]{
Christopher Bose
\and
Gary Froyland
\and
Cecilia Gonz\'alez-Tokman
\and
Rua Murray
}
\address[C. Bose]{Department of Mathematics and Statistics, University of Victoria, Victoria, BC, Canada, V8W 3R4}
\address[G. Froyland and C. Gonz\'alez-Tokman]{School of Mathematics and Statistics,
University of New South Wales, Sydney, NSW, 2052, Australia}
\address[R. Murray]{Department of Mathematics and Statistics, University of Canterbury, Private Bag 4800, Christchurch 8140, New Zealand}

\begin{document}

\begin{abstract}
Ulam's method is a rigorous numerical scheme for approximating invariant densities of dynamical systems.
The phase space is partitioned into connected sets and an 
inter-set transition matrix is computed from the dynamics;  an approximate invariant density is read off as the leading left eigenvector of this matrix.
When a hole in phase space is introduced, one instead searches for \emph{conditional} invariant densities and their associated escape rates.
For Lasota-Yorke maps with holes we prove that a simple adaptation of the standard Ulam scheme provides convergent sequences of escape rates (from the leading eigenvalue), conditional invariant densities (from the corresponding left eigenvector), and quasi-conformal measures (from the corresponding right eigenvector).
We also immediately obtain a convergent sequence for the invariant measure supported on the survivor set.
Our approach allows us to consider relatively large holes. We illustrate the approach with several families of examples,
including a class of Lorenz maps.
\end{abstract}

\maketitle

\section{Introduction}

Dynamical systems $\hat{T}:I\to I$ typically model complicated deterministic processes on a phase space $I$.
The map $\hat{T}$ induces a natural action on probability measures $\eta$ on $I$ via $\eta\mapsto\eta\circ \hat{T}^{-1}$.
Of particular interest in ergodic theory are those probability measures that are $\hat{T}$-invariant;  that is, $\eta$ satisfying $\eta=\eta\circ \hat{T}^{-1}$.
If $\eta$ is ergodic, then such $\eta$ describe the time-asymptotic distribution of orbits of $\eta$-almost-all initial points $x\in I$.
In this paper, we consider the situation where a ``hole'' $H_0\subsetneqq I$ is introduced and any orbits of $\hat{T}$ that fall into $H_0$ terminate.
The hole induces an \emph{open dynamical system} $T:X_0\to I$, where $X_0=I\setminus H_0$.
Because trajectories are being lost to the hole, in many cases, there is no $T$-invariant probability measure.
One can, however, consider \emph{conditionally invariant} probability measures \cite{PianigianiYorke}, which satisfy $\eta\circ T^{-1}(I)\cdot\eta=\eta\circ T^{-1}$, where $0<\eta\circ T^{-1}(I)<1$ is identified as the \emph{escape rate} for the open system.

We will study $\hat{T}$ drawn from the class of Lasota-Yorke maps: piecewise $C^1$ expanding maps of the interval, such that $|D\hat{T}|^{-1}$ has bounded variation.
The hole $H_0$ will be a finite union of intervals.
In such a setting, because of the expanding property, one can expect to obtain conditionally invariant probability measures that are \emph{absolutely continuous} with respect to Lebesgue measure \cite{Collet96,VdBChernov02,LiveMaume}.
Such conditionally invariant measures are ``natural'' as they may correspond to the result of repeatedly pushing forward Lebesgue measure by $\hat{T}$.
In the next section we will discuss further conditions due to \cite{LiveMaume}  that make this precise: (i) how much of phase space can ``escape'' into the hole, and (ii) the growth rate of intervals that partially escape relative to the expansion of the map and the rate of escape.
These conditions also guarantee the existence of a \emph{unique} absolutely continuous conditionally invariant probability measure (accim).
This accim $\nu$, with density $h$, and its corresponding escape rate $\rho$ are the first two objects that we will rigorously numerically approximate using Ulam's method. Existence and uniqueness results for subshifts of finite type with Markov holes were previously established by Collet, Mart{\'{\i}}nez and Schmitt in \cite{CMS3}; see also \cite{CMS1,CMS2,HY02}.

One may also consider the set of points $X_\infty\subset I$ that never fall into the hole $H_0$.
A probability measure $\lambda$ on $X_\infty$ can be defined as the $n\to\infty$ limit of the accim $\nu$ conditioned on $X_{n}$.
The measure $\lambda$ will turn out to be the unique $\hat{T}$-invariant measure supported on $X_\infty$ and has the form $\lambda=h\mu$, where $h$ is a Lebesgue integrable function and $\mu$ is known as the \emph{quasi-conformal measure} for $\hat{T}$.
We will also rigorously numerically approximate $\mu$ and thus $\lambda$.
Robustness of these objects with respect to Ulam discretizations is essentially due to a \textit{quasicompactness} property, 
and a significant part of the paper is devoted to elaborating on this point. 

Our main result, Theorem~\ref{MainThm}, concerns convergence properties of an extension of the well-known construction of Ulam \cite{Ulam}, which allows for efficient numerical estimation of invariant densities of closed dynamical systems.
The Ulam approach partitions the domain $I$ into a collection of connected sets $\{I_1,\ldots,I_k\}$ and computes single-step transitions between partition sets, producing the matrix 
\begin{equation}
\label{ClosedUlameqn}
\hat{P}_{ij}=\frac{m(I_i\cap \hat{T}^{-1}I_j)}{m(I_j)}.
\end{equation}
Li \cite{Li} demonstrated that the invariant density of Lasota-Yorke maps can be $L^1$-approximated by step functions obtained directly from the leading left eigenvector of $\hat{P}$.
Since the publication of \cite{Li} there have been many extensions of Ulam's method to more general classes of maps, including expanding maps in higher dimensions \cite{Ding96,Murray01}, uniformly hyperbolic maps \cite{Froyland95,Froyland99b}, nonuniformly expanding interval maps \cite{Murray10,FroylandMurrayStancevic11}, and random maps \cite{Froyland99,IslamGoraBoyarsky}.
Explicit error bounds have also been developed, eg.\ \cite{Murray98, Froyland99,BoseMurray01}.

We will show that in order to handle open systems, the definition of $\hat{P}$ above need only be modified to $P$, having entries
\begin{equation}
\label{Ulameqn}
P_{ij}=\frac{m(I_i\cap X_0\cap \hat{T}^{-1}I_j)}{m(I_j)}.
\end{equation}
As in the closed setting, one uses the leading left eigenvector to produce a step function that approximates the density $h$ of the accim $\nu$.
However, in the open setting, the leading eigenvalue of $P$ also approximates the escape rate $\rho$ of $\nu$, and the \emph{right} eigenvector approximates the quasi-conformal measure $\mu$.
Note that for closed systems, $\rho=1$ and $\mu=m$.

The literature concerning the analysis of Ulam's method is now quite                                            
large. Early work on Ulam's method for Axiom A repellers \cite{Froyland99b} showed                                         
convergence of an Ulam-type scheme using Markov partitions for the                                          
approximation of  pressure and equilibrium states with respect to the                                       
potential $-\log| \det D \hat{T}|_{E^u}|$. These results apply to the present                                          
setting of Lasota-Yorke maps {\em provided the hole is Markov and                                           
projections are done according to a sequence of Markov partitions.\/}
Bahsoun \cite{Bahsoun} considered non-Markov Lasota-Yorke maps with non-Markov holes and rigorously proved an Ulam-based approximation result for the escape rate.
Bahsoun used the perturbative machinery of \cite{KellerLiverani98}, treating the map $T$ as a small deterministic perturbation of the closed map $\hat{T}$.
In contrast, we apply the perturbative arguments of \cite{KellerLiverani98} directly to the open map, considering the Ulam discretization as a small perturbation of $T$. The advantage of this approach is that we can obtain approximation results whenever the existence results of \cite{LiveMaume} apply. The latter
make assumptions on the expansivity of $T$ (large enough), the escape rate (slow enough), and the rate of generation of ``bad'' subintervals (small enough). From these assumptions we construct an improved Lasota-Yorke inequality that allows us to get tight enough  constants to make applications plausible. 
Besides estimating the escape rate, we obtain rigorous $L^1$-approximations of the accim and approximations of the quasi-conformal measure that converge weakly to $\mu$.
We can treat relatively large holes.

An outline of the paper is as follows.
In Section~\ref{S:setup} we introduce the Perron-Frobenius operator $\mathcal{L}$, formally define admissible and Ulam-admissible holes, and develop a strong Lasota-Yorke inequality.
Section \ref{S:results} introduces the new Ulam scheme and states our main Ulam convergence result.
Section~\ref{S:ex} discusses some specific example maps in detail.
Proofs are presented in Section~\ref{S:pfs}.

\section{Lasota-Yorke maps with holes}\label{S:setup}

The following class of interval maps with holes was studied by Liverani and Maume-Deschamps in \cite{LiveMaume}.

\begin{defn}\label{def:LYmap}
Let $I=[0,1]$.
We call $\hat{T}: I \circlearrowleft$ a \textit{Lasota-Yorke map} if
$\hat{T}$ is a piecewise $C^1$ map, with finite monotonicity partition
\footnote{Throughout this paper, a monotonicity partition $\mcz$ refers to                                    
a partition such that for every $Z\in \mcz$                                                                   
$\hat{T}|_Z$ has a $C^1$ extension to  $\bar{Z}$.}
 $\mcz$, there exists $\hat{\Theta}<1$ such that $\|D\hat{T}^{-1}\|_\infty\leq \hat{\Theta}$, and $\hat{g}:=|D\hat{T}|^{-1}$ has bounded variation.

The \textit{transfer operator} for the map $\hat{T}$ is the bounded linear operator $\hat{\mcl}$, acting on the space $BV$ of functions of bounded variation on $I$, defined by 
\[
  \hat{\mcl} f(x)=\sum_{\hat{T}(y)=x} f(y)\hat{g}(y).
\]
\end{defn}

\begin{defn}
Let $\hat{T}: I \circlearrowleft$ be a Lasota-Yorke map. Let $H_0\subsetneq I$ be a finite union of closed intervals, and let $X_0=I\setminus H_0$. Let $T:X_0 \to I$  be the restriction $T=\hat{T}|_{X_0}$. Both $T$ and the pair $T_0=(\hat{T}, H_0)$ are referred to as open Lasota-Yorke maps (or briefly, open systems), 
and their associated \textit{transfer operator} is the bounded linear operator $\mcl:BV \circlearrowleft$ given by
\[
  \mcl (f)=\hat{\mcl}(1_{X_0}f).
\]
\end{defn}
For each $n\geq 1$, let $X_n=\bigcap_{j=0}^n \hat{T}^{-j}X_0$. Thus, $X_{n}$ is the set of points that have not escaped by time $n$. Also, we denote by $T^n$ the function $\hat{T}^n|_{X^{n-1}}$. 
One can readily check that
\[
  \mcl^n (f)=\hat{\mcl}^n(1_{X_{n-1}}f).
\]

\begin{defn}
Let $T$ be an open Lasota-Yorke map.
A probability measure $\nu$ supported on $X_0\subset I$ which is absolutely continuous with respect to Lebesgue measure and has density $h=\frac{d\nu}{dm}$ of bounded variation is called an \textit{absolutely continous conditional invariant measure} (accim) for $T$ if $\mcl h=\rho h$ for some $0<\rho\leq 1$.

A probability measure $\mu$ on $I$ which satisfies $\mu(\mcl f)=\rho \mu(f)$ for every function of bounded variation $f:I \to \R$, with $\rho$ as above, is called a \textit{quasi-conformal measure} for $T$. 
\end{defn}

\begin{rmk}\label{rmk:accimAndQCMeasure}
It is usual to define $\nu$ to be an accim if $\nu(A)=\frac{\nu(\hat{T}^{-n}A\cap X_{n})}{\nu(X_n)}$ for every $n\geq 0$ and Borel measurable set $A\subset I$. The definitions are indeed equivalent; see \cite[Lemma 1.1]{LiveMaume} for a proof.
The same lemma shows that if $\mu$ is a quasi-conformal measure for $T$, then $\mu$ is necessarily supported on
$X_\infty=\bigcap_{n\geq 0} \bar{X}_n$.
It is also usual to require $\mu$ to satisfy $\mu(\mcl f)=\rho \mu(f)$ for continuous functions only. We will see this makes no difference in our setting, as this weaker requirement implies the stronger one in the previous definition.
\end{rmk}

\subsection{Admissible holes and quasi-invariant measures}\label{subS:admissibility}

As in the work of Liverani and Maume-Deschamps \cite{LiveMaume}, we impose some conditions on the open system in order to be able to analyze it. 
Let us fix some notation.

Let $(\hat{T}, H_0)$ be an open Lasota-Yorke map, which we also refer to as $T$.
For each $n\geq 1$, let $D_n=\{ x\in I : \mcl^n 1(x)\neq 0 \}$, and let $D_\infty:=\bigcap_{n\geq 1} D_n$.
In what follows, we assume that $D_\infty \neq \emptyset$.

For each $\ep>0$ (not necessarily small), we let  $\mc{G}_\ep=\mc{G}_\ep(T)$ be the collection of finite partitions of $I$ into intervals such that $\mc{Z}_\ep\in \mc{G}_\ep(T)$ if (i) the interior of each $A\in \mc{Z}_\ep$ is either disjoint from or contained in $X_0$, and (ii) for each $A\in \mc{Z}_\ep$, 
$\var_A \big(1_{X_0}|DT^{-1}| \big) < \|DT^{-1}\|_\infty (1+\ep)$. Since $H_0$ consists of finitely many intervals, this condition is possible to achieve, as the work of Rychlik \cite[Lemma 6]{Rychlik} shows. We call $\mc{G}_\ep$ the collection of $\ep$-\textit{adequate} partitions (for $T$).
The set of elements of $\mc{Z}_\ep$ whose interiors are contained in $X_0$ is denoted by $\mc{Z}_\ep^*$.
Next, the elements of $\mc{Z}_\ep^*$ are divided into \textit{good} and \textit{bad}. A set $A\in\mc{Z}^*_\ep$ is good if
\[
  \lim_{n\to\infty} \inf_{x\in D_n} \frac{\mcl^{n}1_A(x)}{\mcl^{n}1(x)}>0.
\]
We point out that it is shown in \cite{LiveMaume} that the limit above always exists, as the sequence involved is increasing and bounded, and it is clearly non-negative. The set $A$ is called bad when the limit above is 0.  
We let 
\begin{align*}
\mc{Z}_{\ep,g}&=\{A\in \mc{Z}_\ep^*: A \text{ is good}\}, \text{ and}\\
\mc{Z}_{\ep,b}&=\{A\in \mc{Z}_\ep^*: A \text{ is bad}\}.
\end{align*}
Finally, two elements of $\mc{Z}_\ep^*$ are called \textit{contiguous} if there are no other elements of $\mc{Z}_\ep^*$ in between them (but there may be elements of $\mc{Z}_\ep$ that are necessarily contained in $H_0$).
We let $\xi_\ep=\xi_\ep(T)$ be the infimum over $\ep$-adequate partitions for $T$ of the maximum number of contiguous elements in $\mc{Z}_{\ep,b}$.

In a similar manner, we let $\mc{G}_\ep^{(n)}=\mc{G}_\ep^{(n)}(T)$ be the collection of finite partitions of $I$ into intervals such that $\mc{Z}_\ep^{(n)} \in \mc{G}_\ep^{(n)}$ if (i) the interior of each $A\in \mc{Z}_\ep^{(n)}$ is either disjoint from or contained in $X_{n-1}$, and (ii) for each $A\in \mc{Z}_\ep^{(n)}$, 
$\var_A |1_{X_{n-1}}(DT^n)^{-1}| < \|(DT^n)^{-1}\|_\infty (1+\ep)$.
The partitions
$\mc{Z}_{\ep}^{*(n)}, \mc{Z}_{\ep,g}^{(n)}, \mc{Z}_{\ep,b}^{(n)}$ are defined analogously. We denote by $\xi_{\ep,n}=\xi_{\ep,n}(T)$ the infimum over $\ep$-adequate partitions for $T^n$ of the maximum number of contiguous elements in $\mc{Z}_{\ep,b}^{(n)}$; so $\xi_\ep=\xi_{\ep,1}$.

The following quantities are relevant in what follows:
\begin{align*}
\rho&=\rho(T) :=\lim_{n\to\infty} \inf_{x\in D_n} \frac{\mcl^{n+1}1(x)}{\mcl^{n}1(x)},\\
\t{\Theta}&= \t{\Theta}(T):= \exp \Big(\lim_{n\to \infty} \frac1n \log \|  (DT^n)^{-1} \|_{\infty}\Big), \\
\t{\xi_\ep}&=\t{\xi_\ep}(T) := \exp\Big( \limsup_{n\to \infty} \frac1n \log (1+\xi_{\ep,n}) \Big),
\end{align*}
\begin{align}\label{defnAl_ep}
  \al_\ep&=\al_\ep(T):=\|DT^{-1}\|_\infty (2+\ep+\xi_\ep).
\end{align}

\begin{defn}[Admissible holes]\label{defn:AdmissibleHole}
Let $\hat{T}: I\circlearrowleft$ be a Lasota-Yorke map, and $\ep>0$.
We say that $H_0\subset I$ is:
\begin{itemize}
\item an $\ep$-\textit{admissible hole for $\hat{T}$} if 
$D_\infty \neq \emptyset$ and $\t{\xi_\ep} \t{\Theta}<\rho$,
\item an \textit{admissible hole for $\hat{T}$} if it is $\ep$-admissible for $\ep=1$\footnote{This is the choice made in \cite{LiveMaume}.}.
  \item an $\ep$-\textit{Ulam-admissible hole for $\hat{T}$} if 
$D_\infty \neq \emptyset$ and $\al_\ep<\rho$.
\end{itemize}
\end{defn}

The main result of Liverani and Maume-Deschamps \cite{LiveMaume} is concerned with the existence of the objects we intend to rigorously numerically approximate.

\begin{thm}[{\cite[Theorem A \& Lemma 3.10]{LiveMaume}}]\label{thm:LiveMaume}
Assume $(\hat{T}, H_0)$ is an open system with an admissible hole. Then, 
\begin{enumerate}
\item
There exists a unique absolutely continuous conditionally invariant measure (accim) $\nu=hm$ for $(\hat{T}, H_0)$.
\item
There exists a unique quasi-conformal measure $\mu$ for $(\hat{T}, H_0)$, such that $\mu(\mcl f)=\rho \mu(f)$ for every  $f\in BV$. Furthermore, this measure is atom-free, and satisfies the property that
$$\mu(f)=\lim_{n\to\infty} \inf_{x\in D_n} \frac{\mcl^{n}f(x)}{\mcl^{n}1(x)}$$
for every $f\in BV$, and $\rho=\mu(\mcl 1)$.
\item
The measure $\lam=h\mu$ is, up to scalar multiples, the only $T$ invariant measure supported on $X_\infty$ and absolutely continuous with respect to $\mu$. 
\item
There exist $\ka<1$ and $C>0$ such that for any function of bounded variation $f$,
\[
\Big\|  \frac{\mcl^n f}{\rho^n} - h \mu(f) \Big\|_{\infty} \leq C \ka^n \|f\|_{BV}.
\]
\end{enumerate}
\end{thm}

\begin{rmk}
It follows readily from the proof of Theorem~\ref{thm:LiveMaume}~\cite{LiveMaume} that the same conclusion can be obtained if the hypothesis of $H_0$ being an admissible hole is replaced by $H_0$ being an $\ep$-admissible hole for some $\ep>0$.
\end{rmk}

To close this section, we present a lemma concerning admissibility of different holes, obtained by enlarging an initial hole $H_0$ to $H_{m}:=I\setminus X_{m}$.
This broadens the applicability of Theorem~\ref{MainThm} because enlarging the holes may reduce the number of contiguous bad intervals, and also reduce the variation remaining on the domain of the open Lasota-Yorke map without increasing the expansion.

\begin{lem}[Enlarging holes]\label{lem:EnlargingHoles}
Let  $T_0=(\hat{T},H_0)$ be an open system with an $\ep$-admissible hole, and for each $m\geq 0$, let $H_{m}:=I\setminus X_{m}$. Then, for each $m\geq 0$,
 $T_m:=(\hat{T},H_m)$ is an open system with an $\ep$-admissible hole. Furthermore, let $\rho(T_m), h(T_m)$ and $\mu(T_m)$ be the escape rate, accim and quasi-conformal measures of  $T_m$, respectively. Then we have the following.
 \begin{enumerate}
 \item $\rho(T_m)=\rho(T_0)$,
 \item $\hat{\mcl}^m(h(T_m))=\rho(T_0)^m h(T_0)$, and
\item $\mu(T_m)=\mu(T_0)$.
 \end{enumerate}
\end{lem}
The proof of Lemma~\ref{lem:EnlargingHoles} is presented in \S\ref{pflem:EnlargingHoles}.

\section{Ulam's method for Lasota-Yorke maps with holes}\label{S:results}

\subsection{The Ulam scheme}\label{sec:Ulam}
In the case of a \textit{closed} system $\hat{T}$, the well-known Ulam method introduced in \cite{Ulam} provides a way of approximating the transfer operator with a sequence of finite-rank operators $\hat{\mcl}_k$ defined as in e.g.~\cite{Li}, each coming from discretizing the interval $I$ into $k$ bins (which may or may not be of equal length). The only requirements are that each bin is a non-trivial interval, and that the maximum diameter of the partition elements, denoted by $\tau_k$, goes to 0 as $k$ goes to infinity. We call such $k$-bin partition $\mc{P}_k$. 
The operator $\hat{\mcl}_k$ preserves the $k-$dimensional subspace
$\text{span}\{\chi_j: \chi_j= 1_{I_j}, I_j\in\mathcal{P}_k\}$.
The matrix $\hat{P}_k$ defined in the introduction represents the action of $\hat{\mcl}_k$
on this space,  with respect to the ordered basis $(\chi_1, \dots, \chi_k)$ \cite{Li}.

In the case of an \textit{open} system $(\hat{T}, H_0)$, one can still follow Ulam's approach to define a discrete approximation $\mcl_k$ to the transfer operator $\mcl$.
For a function $f\in BV$, the operator is defined by $\mcl_k(f)=\pi_k(\mcl f)=\pi_k\hat{\mcl}(1_{X_0}f)$, where $\pi_k$ is given by the formula 
\[
\pi_k(f)=\sum_{j=1}^k \frac{1}{m(I_j)}\Big(\int \chi_j\,f\,dm\, \Big)\chi_j.
\]
The entries of the Ulam transition matrix $P_k$ representing $\mcl_k$ in the ordered basis $(\chi_1, \dots, \chi_k)$
are
\[
(P_k)_{ij}=\frac{m(I_i\cap X_0 \cap \hat{T}^{-1}I_j )}{m(I_j)}.
\]
(When the partition $\mc{P}_k$ is uniform\footnote{That is,  $m(I_i)=m(I_j), \forall i,j$.}, the transition matrices $\hat{P}_k$ defined in \eqref{ClosedUlameqn} are stochastic, and $P_k$ are substochastic, the loss of mass being a consequence of the presence of a hole.)
Since the entries of $P_k$ are non-negative, an extension of the Perron-Frobenius theorem applies and provides the existence of a non-negative eigenvalue $0\leq \rho_k\leq 1$ of maximal absolute value for $P_k$, with associated left and right eigenvectors with non-negatives entries; see e.g. \cite{BermanPlemmons}. In general, these may or may not be unique. Non-negative left eigenvectors $\mathbf{p}_k$ of                                                          
 $P_k$ induce densities on $I$ according to the formula                                                                     
 $$h_k=\sum_{j=1}^k [\mathbf{p}_k]_j \chi_j,$$ 
 (where we adopt the                                                            
 convention that a vector $\mathbf{x}$ can be written in component form                                                     
 as $\mathbf{x}=([\mathbf{x}]_1,\ldots, [\mathbf{x}]_k)$. Non-negative                                                      
 right eigenvectors ${\psi_k}$ of $P_k$ induce measures $\mu_k$ on $I$                                                       according to the formula 
\[
  \mu_k(E) = \sum_{j=1}^k [\psi_k]_j\,m(I_j\cap                                                               
 E).
\]

We conclude the section with the following.
\begin{lem}\label{lem:QCMeasureFromEvector}
Let $P_k$ be the                                                       
 matrix representation of $\mcl_k=\pi_k\circ\mcl$ with respect to the basis                                                     
 $\{\chi_j\}$. If $P_k\psi_k = \rho_k \psi_k$ then the measure~$\mu_k$                                                      
 corresponding to $\psi_k$                                                                                                  
 satisfies $\mu_k(\mcl_k\pi_k\varphi) = \rho_k\mu_k(\varphi)$ for every                                                       
 $\varphi\in L^1(m)$. 
\end{lem}
\begin{proof}
Let $\varphi\in L^1(m)$ and put $\varphi_k=\pi_k\varphi$. Then,
\begin{align*}
\mu_k(\varphi) &=\int\varphi\,d\mu_k = \sum_{j=1}^k \int_{I_j}\varphi\,dm\,[\psi_k]_j
= \sum_{j=1}^k \int_{I_j}\pi_k\varphi\,dm\,[\psi_k]_j\\
&= \sum_{j,j'=1}^k\int_{I_j}\varphi_k\,dm\,(P_k)_{jj'}[\psi_k]_{j'}(\rho_k)^{-1}
= \sum_{j'=1}^k\int_{I_{j'}}\Lk\varphi_k\,dm\,[\psi_k]_{j'}(\rho_k)^{-1}\\
&=(\rho_k)^{-1}\,\int \Lk \varphi_k\,d\mu_k={\rho_k}^{-1}\mu_k(\Lk \varphi_k),
\end{align*}
where the last equality in the second line follows from the fact that $P_k$ is the  matrix representing $\mcl_k$ in the basis $\{\chi_j\}$, and acts on densities by right multiplication (i.e. if $\mathbf{p}$ is the vector representing the  function $\varphi_k$, then $\mathbf{p}^T P_k$ is the vector representing $\mcl_k \varphi_k$).
\end{proof}

\subsection{Statement of the main result}

The main result of this paper is the following. Its proof is presented in \S\ref{sec:pfMainThm}.
\begin{thm}\label{MainThm}
Let $\hat{T}: I \circlearrowleft$ be a Lasota-Yorke map with an $\ep$-Ulam-admissible hole $H_0$. Let $h \in BV$ be the unique accim for the open system $(\hat{T}, H_0)$, and $\mu$ the unique quasi-conformal measure for the open system supported on $X_\infty$, as guaranteed by Theorem~\ref{thm:LiveMaume}. Let $\rho$ be the associated escape rate. 
For each $k\in \N$, let $\rho_k$ be the leading eigenvalue of the Ulam matrix $P_k$. Let $h_k$ be densities induced from non-negative left  eigenvectors of $P_k$ corresponding to $\rho_k$. Let $\mu_k$ be measures induced from non-negative right eigenvectors of $P_k$ corresponding to $\rho_k$. 
Then, 
\begin{enumerate}[(I)]
\item\label{it:I} For $k$ sufficiently large, $\rho_k$ is a simple eigenvalue for $P_k$.
Furthermore, $\lim_{k\to \infty}\rho_k=\rho$, and there exists $\eta\in (0,1)$
\footnote{In fact, any $\eta<\frac{\log \rho/\be}{-\log\be}$ with $\rho>\be>\al_\ep$ is valid.}
such that $|\rho_k-\rho|\leq O({\tau_k}^\eta)$,
where $\tau_k:=\max_{I_j\in \mc{P}_k} m(I_j)$ is the maximum diameter of the elements of $\mc{P}_k$.
\item\label{it:II}  $\lim_{k\to \infty}h_k=h$ in $L^1(m)$.
\item\label{it:III} $\lim_{k\to \infty}\mu_k=\mu$ in the weak-* topology of measures. 
Furthermore, for every sufficiently large $k$, $\spt(\mu)\subseteq \spt(\mu_k)$.
\end{enumerate}
\end{thm}

We will also establish a relation between admissibility and Ulam-admissibility of holes.
\begin{lem}[Admissibility and Ulam-admissibility]\label{lem:AdmissibleAndUlamAdmissible}
If  $H_0$ is an $\ep$-admissible hole for $\hat{T}$, there is some $n\in \N$ such that $H_{n-1}:=I\setminus X_{n-1}$ is $\ep$-Ulam-admissible for $\hat{T}^n$. 
\end{lem}
The proof of this lemma is presented in \S\ref{subS:AdmVsUlamAdm}.
This result, together with Lemma~\ref{lem:EnlargingHoles}, broadens the scope of applicability of Theorem~\ref{MainThm} by allowing to (i) replace the map by an iterate (Lemma~\ref{lem:AdmissibleAndUlamAdmissible}), or (ii) enlarge the hole in a dynamically consistent way (Lemma~\ref{lem:EnlargingHoles}).
It also ensures that several examples in the literature can be treated with our method; in particular, all the examples presented in \cite{LiveMaume}.

\section{Examples}\label{S:ex}

To illustrate the efficacy of Ulam's method, beyond the small-hole setting, we present some examples of Ulam-admissible open Lasota-Yorke systems.
We start with the case of full-branched maps in \S\ref{subS:fullBranchedMaps}, and treat some more general examples, including $\beta$-shifts, in \S\ref{subS:nonfullBranchedMaps}.
We then analyze Lorenz-like maps. They provide transparent evidence of the scope of the results for open systems, as well as closed systems with repellers. They also illustrate how the admissibility hypothesis may be checked in applications.

\subsection{Full-branched maps.}\label{subS:fullBranchedMaps}
In the next examples, we will use the following notation.
Given a  Lasota-Yorke map with holes,  $(\hat{T}, H_0)$ with monotonicity partition  $\mc{Z}$, we let
$\mc{Z}_h=\{Z\in \mc{Z}:  Z \subseteq H_0 \}$,
$\mc{Z}_f=\{Z\in \mc{Z}: Z\cap H_0=\emptyset, T(Z)=I \}$ and
$\mc{Z}_u=\{Z\in \mc{Z}: Z\not\in \mc{Z}_h \cup  \mc{Z}_f\}$. Thus, the elements of $\mc{Z}_f$
are precisely the ones contained in $X_0$ that are full branches for $T$, and those of $\mc{Z}_u$ are the remaining ones.

\begin{defn}
A \textit{full-branched map with holes}, $(\hat{T}, H_0)$, is a Lasota-Yorke map with holes, such that $\mc{Z}_u=\emptyset$.
\end{defn}

For piecewise linear maps, the situation is rather simple.
\begin{lem}\label{rmk:pwLinearFullBranch}
Let $T_0=(\hat{T}, H_0)$ be a piecewise linear full-branched map with holes. Then, for every $\ep>0$ the following holds: $\xi_\ep(T_0)=0$,
\begin{align*}
\rho(T_0) &= 1-\leb(H_0), \quad \text{and}\\
\al_\ep(T_0) &=\max_{Z\in \mc{Z}_f}\leb(Z)(2+\ep).
\end{align*}
\end{lem}
\begin{proof}
If $T_0$ is a piecewise linear full-branched map, then
each interval $Z\in \mc{Z}_f$ is good. 
Observing that an interval being good is equivalent to having non-zero $\mu$ 
measure, and using the fact that $\mu$ is atom-free, each $Z$ may be split 
into two good intervals $Z_-, Z_+$ in such a way that there is at most one 
discontinuity of $g:=\mathbf{1}_{X_0}D\hat{T}^{-1}$ on each $Z_-, Z_+$. 
Thus, $\var_{Z_-}(g), \var_{Z_+}(g)\leq \|DT_0^{-1}\|_{\infty}$. Therefore 
$\xi_\ep(T_0)=0$. Also,
\[
\mc{L}_0(1)(x)=\sum_{y \in Z\in \mc{Z}_f, T_0(y)=x} \frac{1}{|DT_0(y)|}=\sum_{Z\in \mc{Z}_f} \leb(Z)=1-\leb(H_0).
\] 
On the other hand, $\sup_{x\in Z\in \mc{Z}_f}\frac{1}{|DT_0(x)|}=\max_{Z\in \mc{Z}_f}\leb(Z)$.
\end{proof}

In fact, in the piecewise linear, full branched setting, a direct calculation shows that Lebesgue measure is an accim for the open system. For perturbations of these systems, explicit estimates of $\rho$ and $\al_\ep$ are not generally available. However, we have the following bounds.
\begin{lem} \label{lem:FullBrancheMapswHoles}
Let $T_0=(\hat{T}, H_0)$ be a full-branched map with holes.
Then, for every $\ep>0$, there exists some computable $m\in \N$ such that $\xi_\ep(T_m)=0$, where $T_m:=(\hat{T}, H_m)$ is obtained from $T_0$ by enlarging the hole, as in Lemma~\ref{lem:EnlargingHoles}. Furthermore,
\begin{align*}
\rho(T_m)=\rho(T_0) &\geq \inf_{x\in I}\sum_{y \in Z\in \mc{Z}_f, T_0(y)=x} \frac{1}{|DT_0(y)|}=: \rho_0 \quad \text{and}\\
\al_\ep(T_m) &\leq \sup_{x\in Z\in \mc{Z}_f}\frac{1}{|DT_0(x)|}(2+\ep)=:\alpha_{\epsilon,0}.
\end{align*}
\end{lem}
An immediate consequence is the following.
\begin{cor}\label{cor:checkQCFullBranch}
In the setting of Lemma~\ref{lem:FullBrancheMapswHoles}, if $\rho_0>\alpha_{\epsilon,0}$, then $H_m$ is $\ep$-Ulam admissible for $\hat{T}$.
In this case, Lemma~\ref{lem:EnlargingHoles} allows one to approximate the escape rate, accim and quasi-conformal measure for $T_0$
via Theorem~\ref{MainThm} applied to $T_m$.
\end{cor}

\begin{proof}[Proof of Lemma~\ref{lem:FullBrancheMapswHoles}]
First, let us note that for any map with $\mc{Z}_f\neq \emptyset$, we have that $D_\infty\neq \emptyset$, as the map has at least one fixed point outside the hole.
If $m$ is sufficiently large, each interval $Z\in\mc{Z}^{(m)}$ is either (i) contained in $H_{m-1}$, and thus not in $\mc{Z}^{*(m)}$ or (ii) $T_0^m(Z)=I$ and $\var_Z(\hat{g}1_{X_m})<\|\hat{g}1_{X_m}\|_\infty(1+\ep)$. 
In the latter case, $Z$ is a good interval for $T_0$, because
$\mu_0(Z)= \rho_0^m \mu_0(\mcl_0^m 1) \geq  \rho_0^m \|DT_0^m\|_\infty^{-1} \mu_0(I)>0$.
Since good intervals for $T_0$ and for $T_m$ coincide (see                                                                                     
beginning of proof of Lemma~\ref{lem:EnlargingHoles}), we get that $\xi_\ep(T_m)=0$.

Furthermore,
\[
\rho(T_0)=\rho(T_0)\mu_0(1)=\mu_0(\mc{L}_0(1)) \geq  \inf_{x\in I}\mc{L}_0(1)(x)= \inf_{x\in I}\sum_{y \in Z\in \mc{Z}_f, T_0(y)=x} \frac{1}{|DT_0(y)|}.
\]
The bound on $\al_\ep(T_m)$ follows directly from the definition. 
\end{proof}

The following is an interesting consequence of Lemmas~\ref{rmk:pwLinearFullBranch} and~\ref{lem:FullBrancheMapswHoles}.
\begin{cor}\label{cor:pwLinearFullBranched}
Let $(\hat{T}, H_0)$ be a piecewise linear full-branched map with holes.
Assume that $\leb(H_0)< 1-2\max_{Z\in \mc{Z}_f}\leb(Z)$. Then, if $\ep>0$ is sufficiently small, $H_0$ is $\ep$-Ulam-admissible for any full-branched map $(\hat{S},H_0)$ that is a sufficiently small $C^{1+Lip}$ perturbation of $(\hat{T}, H_0)$ (where the $C^{1+Lip}$ topology is defined, for example, by the norm given by the maximum of the $C^{1+Lip}$ norms of each branch).  In particular, Theorem~\ref{MainThm} applies.
\end{cor}
\begin{proof}
The statement for $(\hat{T}, H_0)$ follows from Remark~\ref{rmk:pwLinearFullBranch}.
For perturbations, the statement follows from Lemma~\ref{lem:FullBrancheMapswHoles}, by observing that the quantities $\rho_0$ and $\al_{\ep,0}$, as well as the variation of $1/|D\hat{T}|$ on each interval depend continuously on $\hat{T}$, with respect to the $C^{1+Lip}$ topology.
\end{proof}

Corollary~\ref{cor:pwLinearFullBranched} can apply to maps with arbitrarily large holes, as the next example shows.  
\begin{ex}[Arbitrarily large holes]
Let $\del>0$, $H_0=[\del, 1-\del]$, and 
\[
T_\del(x)= \begin{cases}
\frac 2\del x \quad\text{if } x<\frac \del 2, \\
1-\frac 2\del (x-\frac \del 2) \quad\text{if } \frac \del 2\leq x<\del, \\
\frac 2\del (x-1+\del) \quad\text{if } 1-\del \leq x<1-\frac \del 2, \\
1-\frac 2\del (x-1+\frac \del 2) \quad\text{if } 1-\frac \del 2\leq x\leq 1.
\end{cases}
\]
Then, $\leb(H_0)=1-2\del< 1-\del=1-2\max_{Z\in \mc{Z}_f}\leb(Z)$ and the hypotheses of Corollary~\ref{cor:pwLinearFullBranched} are satisfied. Thus, Ulam's method converges for sufficiently small $C^{1+Lip}$ perturbations of $T_\del$ that are full-branched.
\end{ex}

Other examples of this type may be found in \cite{Bahsoun} and \cite{BahsounBose}. Bahsoun \cite{Bahsoun} established rigorous computable bounds for the errors in the Ulam method, which allowed him to find rigorous bounds on the escape rate for open Lasota-Yorke maps. Bose and Bahsoun \cite{BahsounBose} related the escape rate to the Lebesgue measure of the hole. Both results rely on the existence of Lasota-Yorke type inequalities, relating $BV$ and $L^1(m)$ norms.
Such inequalities may be obtained by exploiting the full-branched structure of the map.
\begin{ex}[Bahsoun  \cite{Bahsoun}]
Let 
\[
\hat{T}(x)= \begin{cases}
2.08 x \quad\text{if } x<\frac12, \\
2-2x \quad \text{if } x\geq \frac12.
\end{cases}
\]
In this case, Corollary~\ref{cor:pwLinearFullBranched} does not yield direct information about the applicability of the Ulam method to perturbations of this map,
because $\leb(H_0)=\frac{ .08}{4.16}$ but $1-2\max_{Z\in \mc{Z}_f}\leb(Z)=0$. However, the hypotheses of the corollary are easily satisfied for the second power of the map.
We note that $\rho$ controls the rate of mass loss, which is slower than 4.08/4.16, while $\alpha_\ep$ is related to the relaxation rate on the survivor set. In this case, one iteration is not enough to see this asymptotic rate is less than the mass loss. However, two iterations suffice.
\end{ex}

\subsection{Nearly piecewise linear maps with enough full branches.}
\label{subS:nonfullBranchedMaps}

When non-full branches are present, the dynamics is typically non-Markovian. Thus, even in the piecewise linear setting there may not be direct ways to                                                                        find the various objects of interest (escape rates, accims and quasi-conformal measures) exactly.
We show that Ulam's method provides rigorous approximations in specific systems. 
The following example is closely related to \cite[6.2 \& 6.3]{LiveMaume}.
\begin{lem}\label{lem:manyFullBranches}
Let $T=(\hat{T}, H_0)$ be a piecewise linear Lasota-Yorke map with holes, and assume $\mc{Z}_f\neq \emptyset$. Let $c_u$ be the maximum number of contiguous elements in $\mc{Z}_u$. If $ \|DT^{-1}\|_\infty (3+ c_u)< \rho$, then $H_0$ is $(1+\ep)$-Ulam-admissible for $\hat{T}$, for every $\ep>0$ sufficiently small. Thus, the hypotheses of Theorem~\ref{MainThm} are satisfied.
\end{lem}
\begin{proof}
For any map with $\mc{Z}_f\neq \emptyset$, we have that $D_\infty\neq \emptyset$, as the map has at least one fixed point outside the hole. Furthermore, for each $Z\in \mc{Z}$, one has that $\var_Z(g)\leq 2\|g\|_{\infty}$, so $\mc{Z}$ is a $(1+\ep)$-adequate partition for $T$. Also, it follows from the definition of $\mc{Z}_g$ that $\mc{Z}_f \subseteq \mc{Z}_g$. Thus, $\mc{Z}_b \subseteq \mc{Z}_u$, and $\xi_{1+\ep}\leq c_u$. Therefore, $\al_\ep\leq \|DT^{-1}\|_\infty (3+ \ep+c_u)<\rho$, provided $\ep>0$ is sufficiently small.
\end{proof}

A concrete example where the previous lemma applies is that of $\beta$-shifts.
\begin{ex}
Let $\beta>1$, and $\hat{T}_\beta$ be the $\beta$-shift, $\hat{T}_\beta(x)=\beta x \pmod{1}$.
Let $H_0\subset I$ be a finite union of closed intervals, and let $f$ be the number of full branches of $\hat{T}_\beta$ outside $H_0$. Then, for the open system $(\hat{T}_\beta, H_0)$, we have that $\rho\geq \frac{f}{\beta}$.
Then, the hypotheses of Lemma~\ref{lem:manyFullBranches} are satisfied, provided 
$f>3+c_u$. This happens, for example, when $\beta\geq 5$ and $H_0$ is a single interval
of the form $[\frac{[\beta]}{\beta},y]$ or $[y,1]$, with $\frac{[\beta]}{\beta}<y<1$. Also,  when $\beta\geq 6$ and $H_0$ is a single interval contained in $[\frac{[\beta]}{\beta},1]$; or when $\beta\geq 7$ and $H_0$ is any interval leaving at least 7 full branches in $X_0$ (recall from Subsection~\ref{subS:admissibility} that two                                                                                    
bad elements of $\mc{Z}_u$ are contiguous if there are no good elements of $\mc{Z}_f \cup \mc{Z}_u$ between                                                                                    
them, but there may be elements of $\mc{Z}_h$ in between).                                                                                                                                                   

We include Figures \ref{fig1a}-\ref{fig1h}, obtained from numerical experiments for $\be=5.9$, and two different choices of holes. 
They include approximations to the densities of accims and cumulative distribution functions of the quasi-conformal measures for systems with a hole, as well as the acim and conformal measure for the closed system. 

\begin{figure}

\centering 
  \includegraphics[width=7cm]{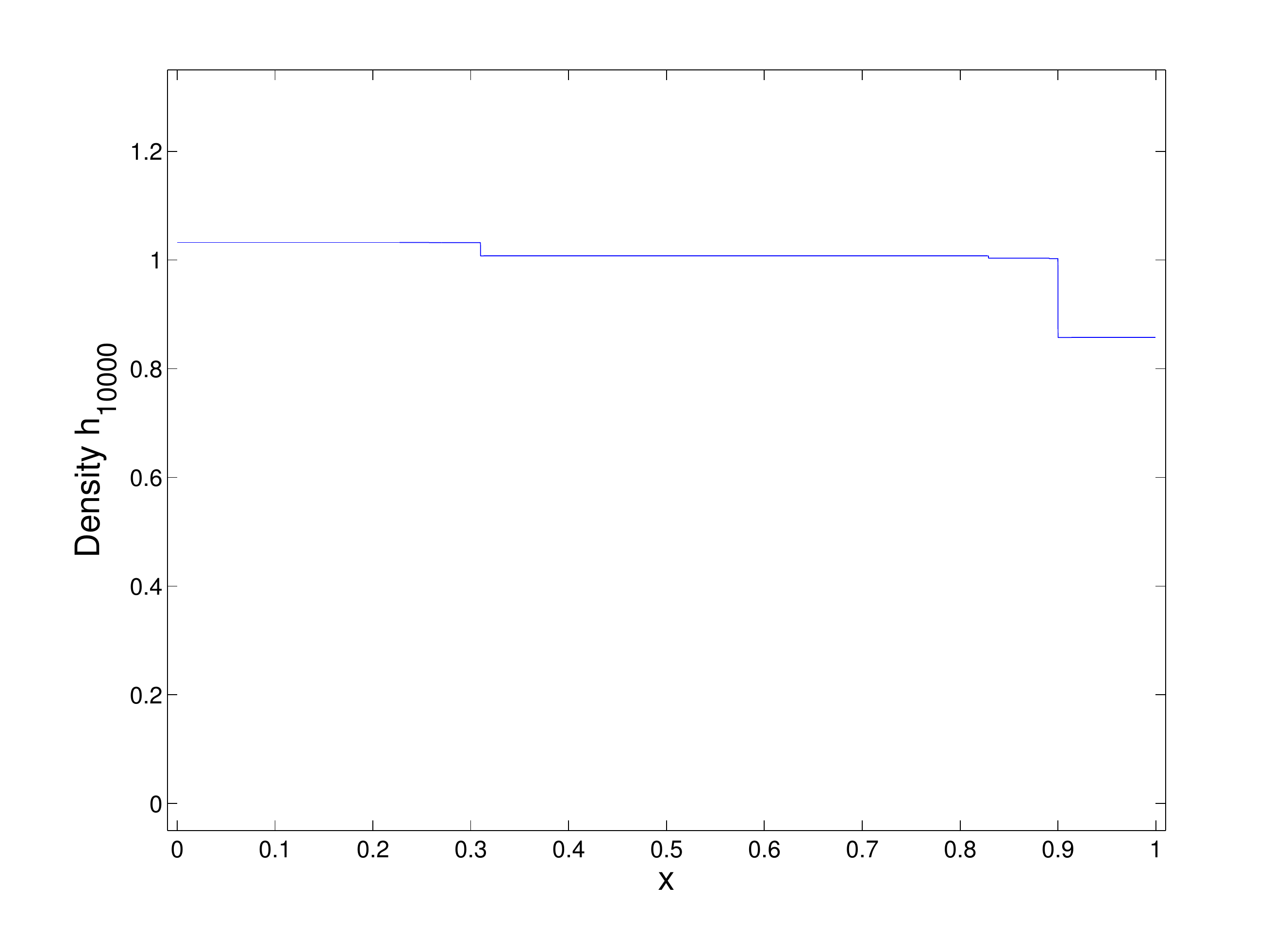}
\qquad\includegraphics[width=7cm]{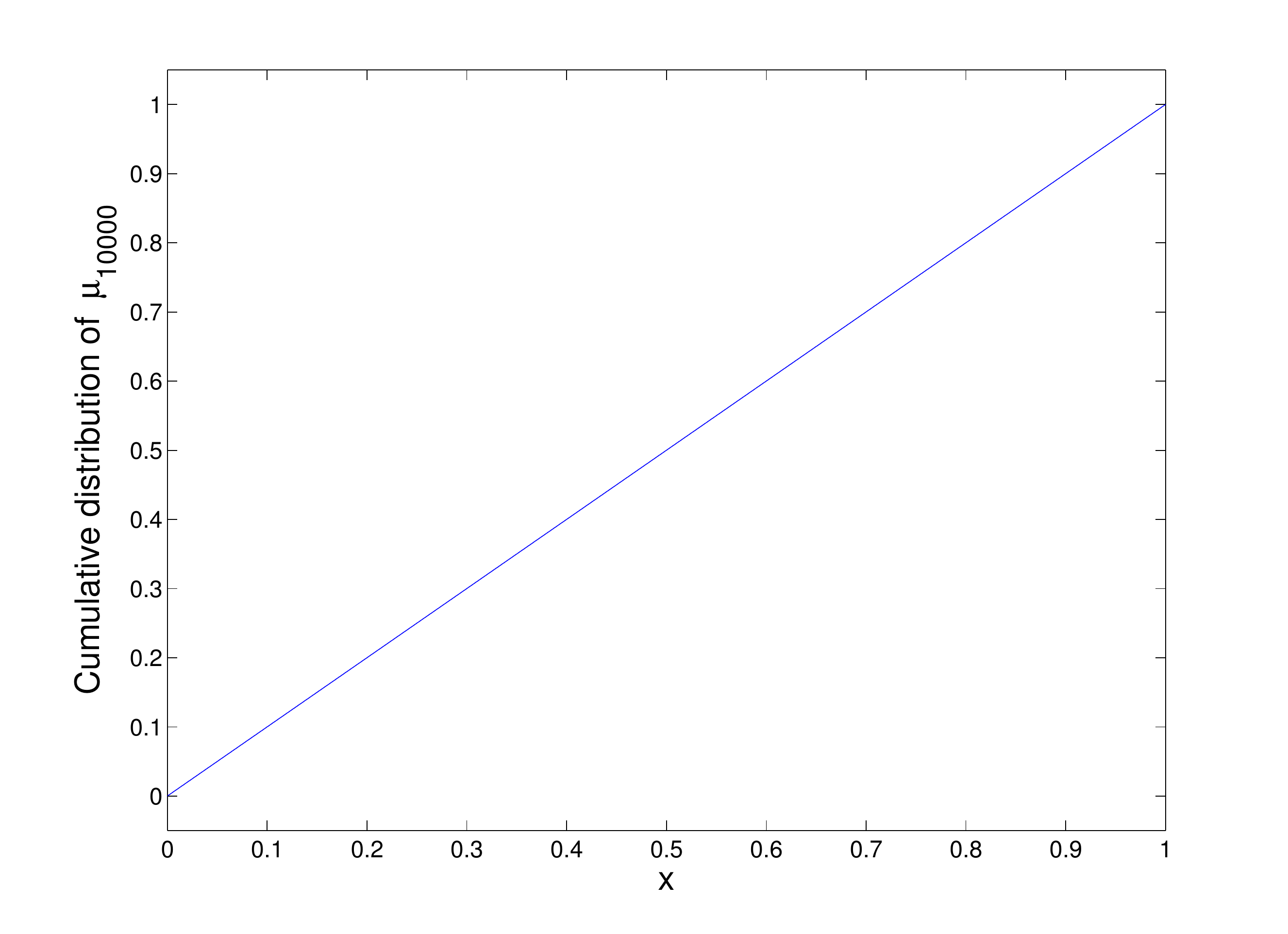}\\

  \caption{Ulam approximations with $k=10000$ for $\hat{T}_\beta$ ($\beta=5.9$) 
and $H_0=\emptyset$. 
Left: Graph of approximate density $h_k$ of acim.
Right: The ``quasi-conformal'' measure, depicted as $\mu_k([0,x])$ vs. $x$.
Note that $\mu_k$ approximates Lebesgue measure on $[0,1]$, 
as $(\hat{T}_\beta,\emptyset)$ is closed.}\label{fig1a}\label{fig1b}
\end{figure}

\begin{figure}

\centering 
\includegraphics[width=7cm]{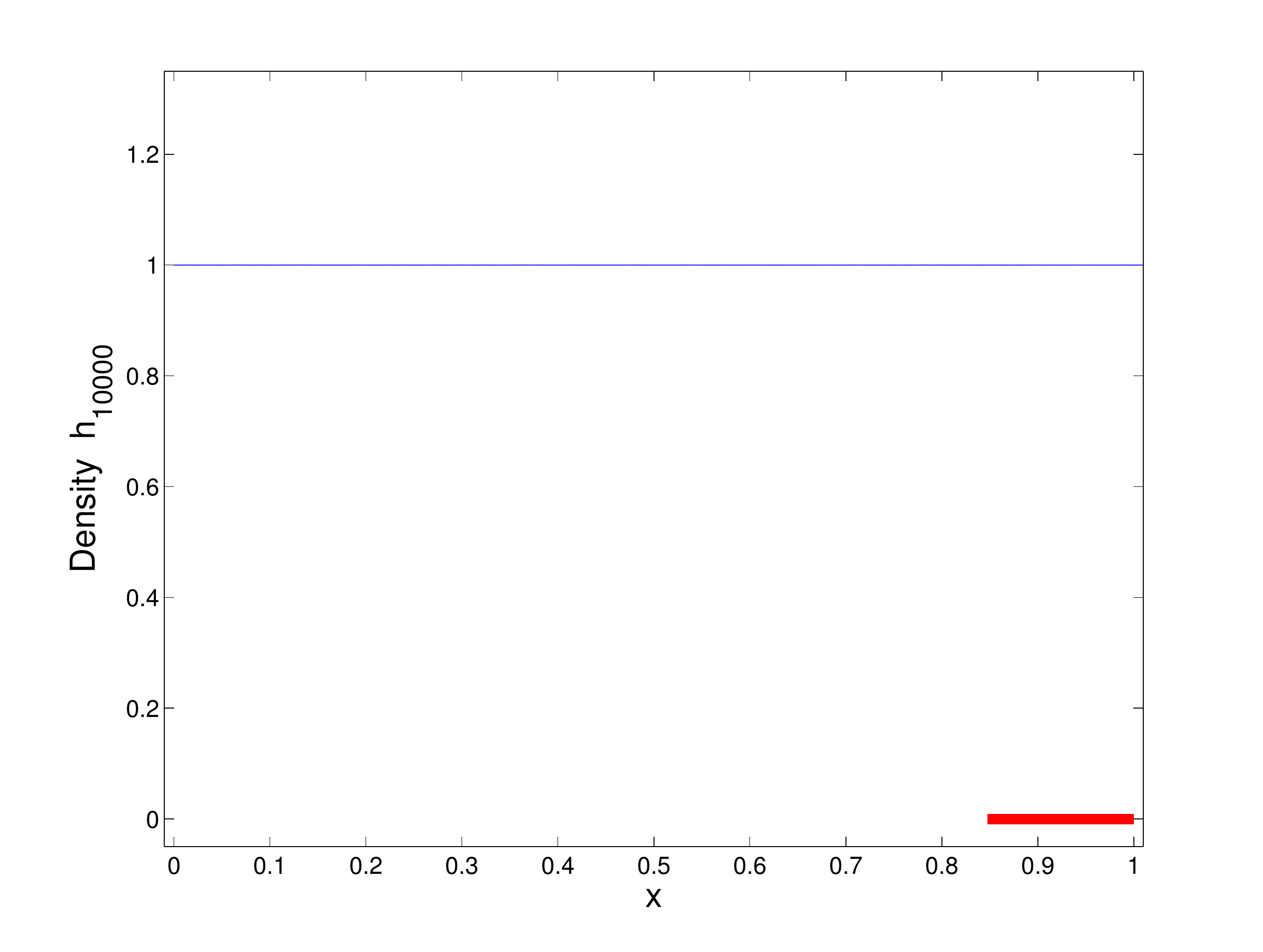}
\qquad\includegraphics[width=7cm]{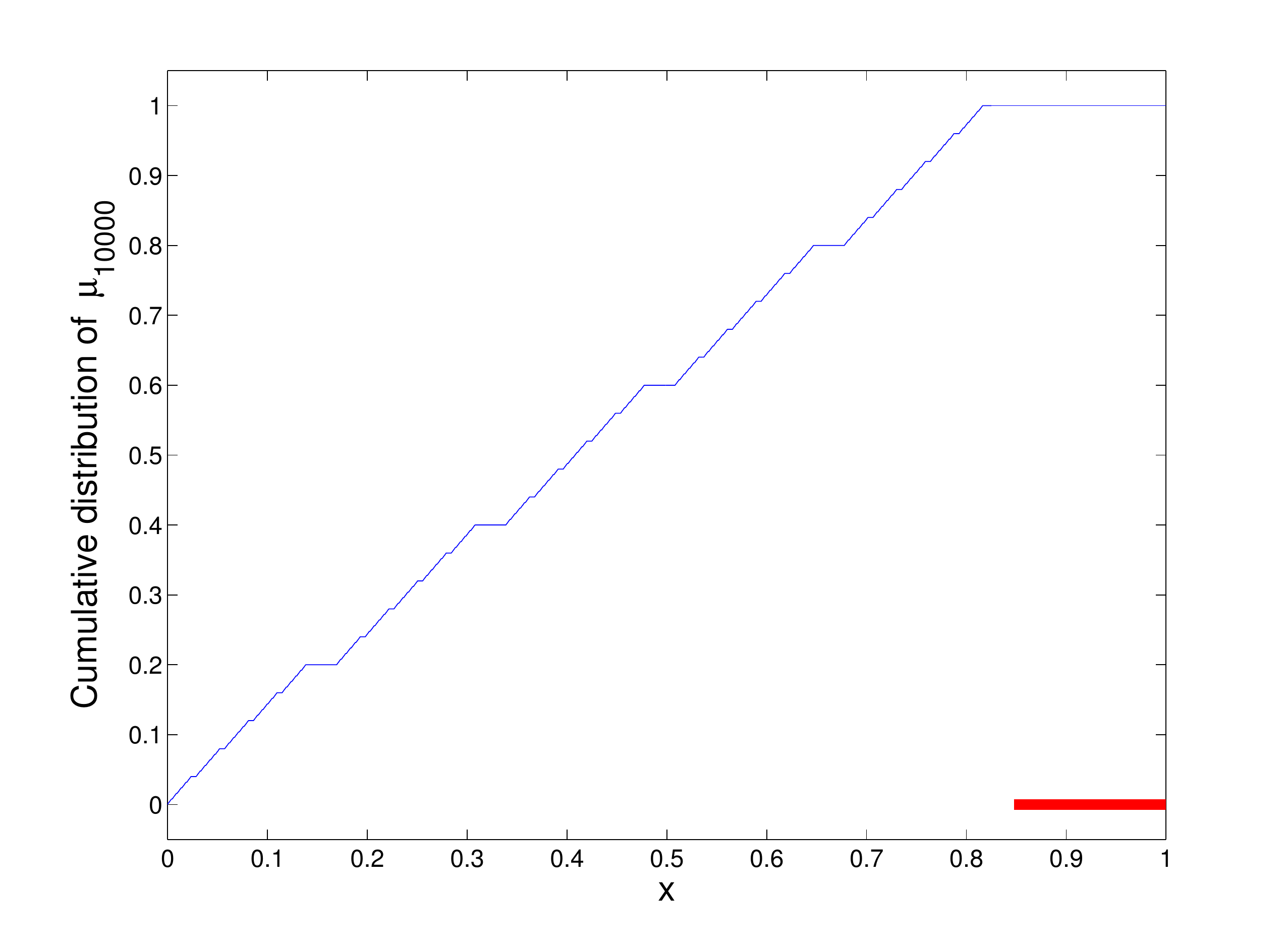}\\

  \caption{Ulam approximations with $k=10000$ for $\hat{T}_\beta$ ($\beta=5.9$) 
and $H_0=[\frac{5}{5.9},1]$ (shown in red). The computed 
value of $\rho_k$ is 0.8475 (4 s.f.), and in fact agrees up to 11 s.f. with
the exact value for $\rho$ (the length of $X_0$: $5/5.9$).
Left: Graph of computed density $h_k$ of accim (note that the function $1$
is a fixed point of both $\mcl$ and $\pi_k\mcl$).
Right: The approximate quasi-conformal measure, 
depicted as $\mu_k([0,x])$ vs. $x$. Note that $\mu_k$ has no support on $H_0$.
}\label{fig1c}\label{fig1d}
\end{figure}

\begin{figure}

\centering \includegraphics[width=7cm]{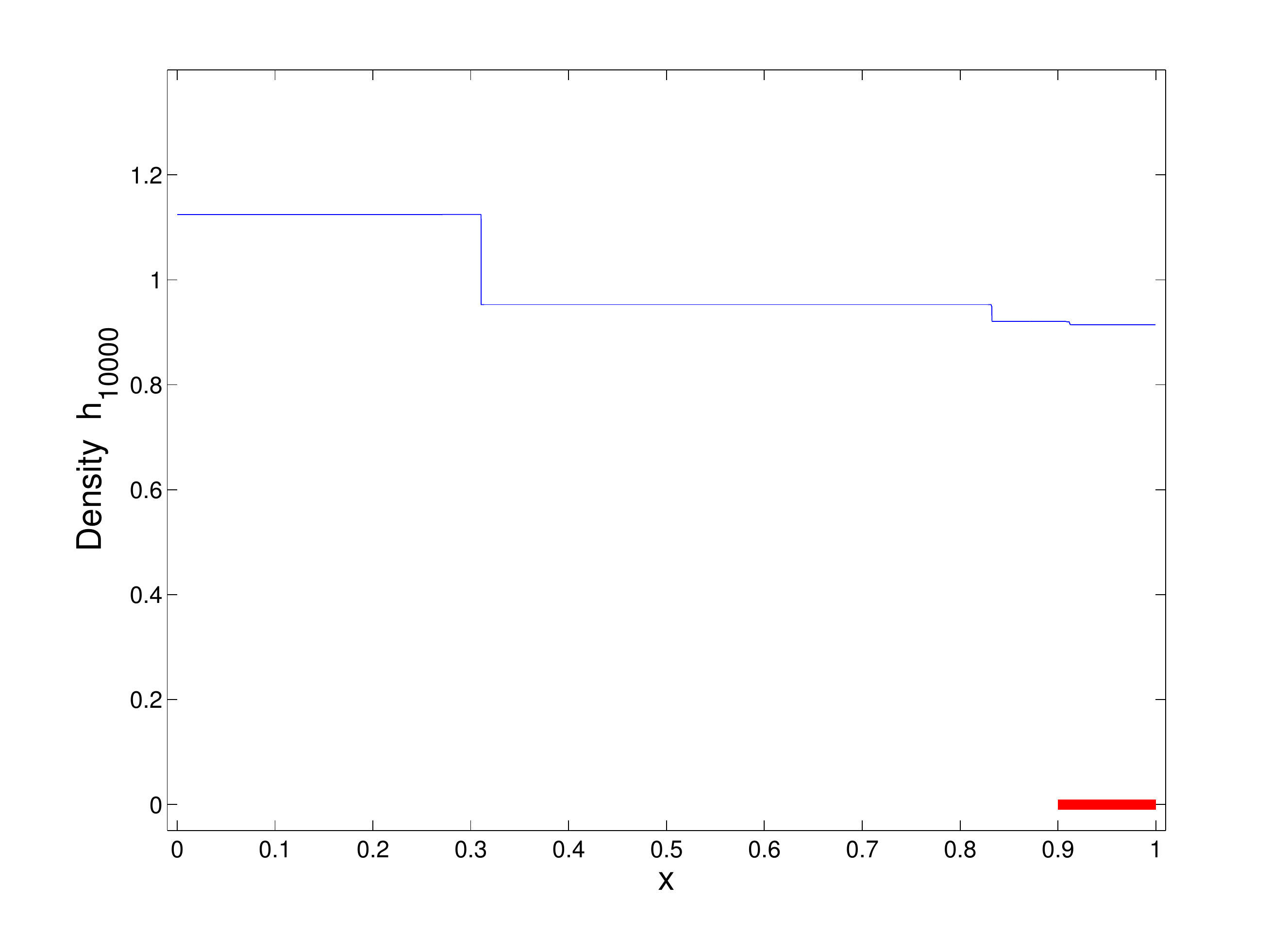}
\qquad\includegraphics[width=7cm]{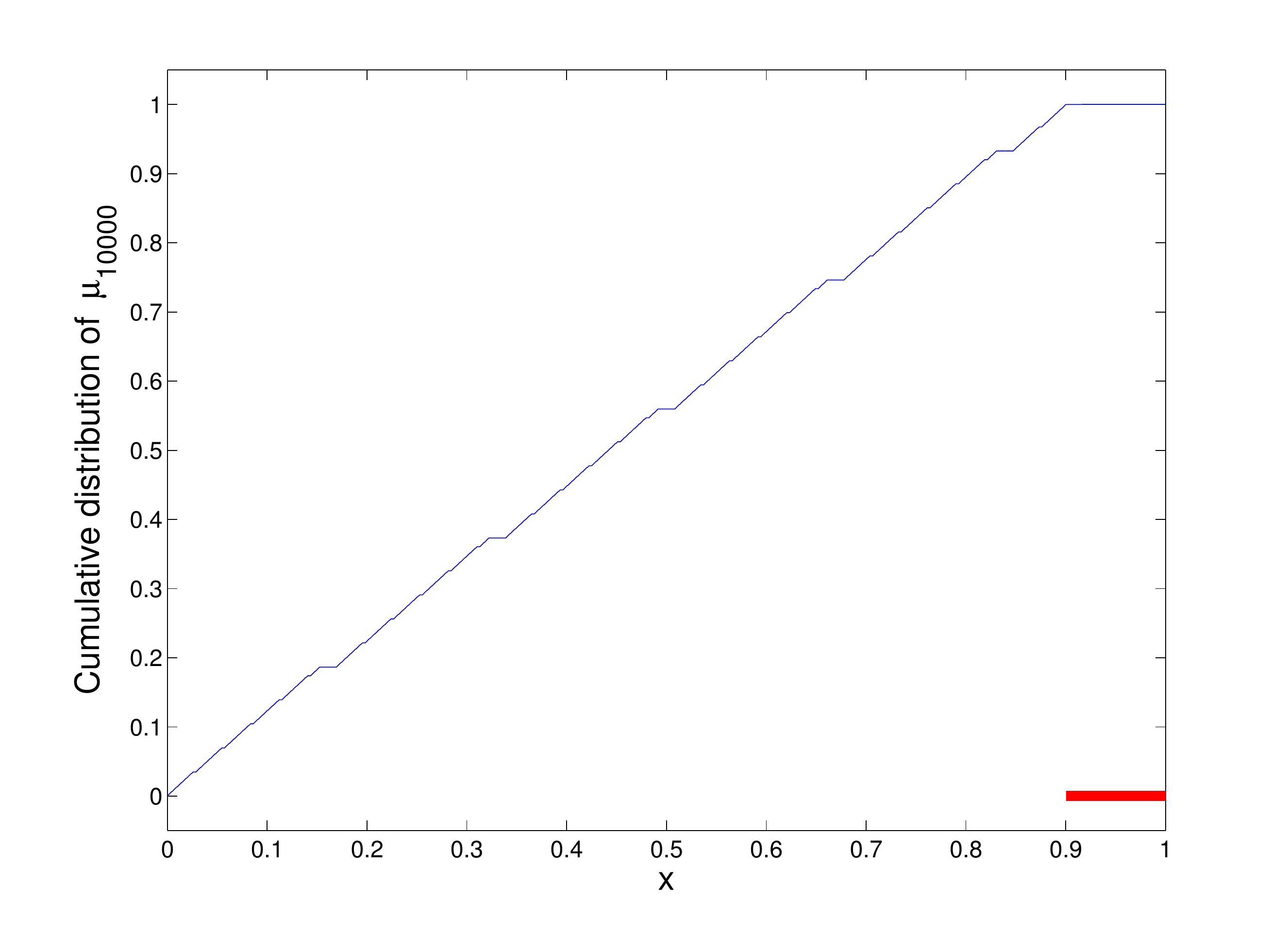}\\

  \caption{Ulam approximations with $k=10000$ for $\hat{T}_\beta$ ($\beta=5.9$) 
and $H_0=[0.9001,1]$ (shown in red). The computed value of $\rho_k$ is 
0.9086 (4 s.f.)
Left: Graph of approximate density $h_k$ of accim.
Right: The approximate quasi-conformal measure, 
depicted as $\mu_k([0,x])$ vs. $x$. Note that $\mu_k$ has no support on $H_0$.}
\label{fig1g}\label{fig1h}
\end{figure}

\end{ex}

\begin{rmk}
Using lower bounds on $\rho$ such as those of Lemma~\ref{lem:FullBrancheMapswHoles}, one can extend the conclusion of Lemma~\ref{lem:manyFullBranches} as in Corollary~\ref{cor:pwLinearFullBranched}, to cover small $C^{1+Lip}$ perturbations of piecewise linear maps that respect the partition $\mc{Z}_h \cup \mc{Z}_f \cup \mc{Z}_u$.
\end{rmk}

\subsection{Lorenz-like maps}\label{s:Lorenzlike}
Let us consider the following two-parameter family of maps of $I=[-1,1]$:
\begin{equation}\label{eq:LorenzMap}
 T_{c,\al}(x)=\begin{cases} cx^{\al}-1 \quad &\text{ if } x>0, \\ 1-c|x|^{\al} \quad &\text{ if } x<0, \end{cases}
\end{equation}
where $c>0, \al \in (0,1)$.
When $c>2$, the system is open and the hole is implicitly defined as $H_{0, c, \al}=T_{c, \al}^{-1}(\R\setminus [-1,1])$.

\begin{figure}
\centering
  \includegraphics[width=10cm]{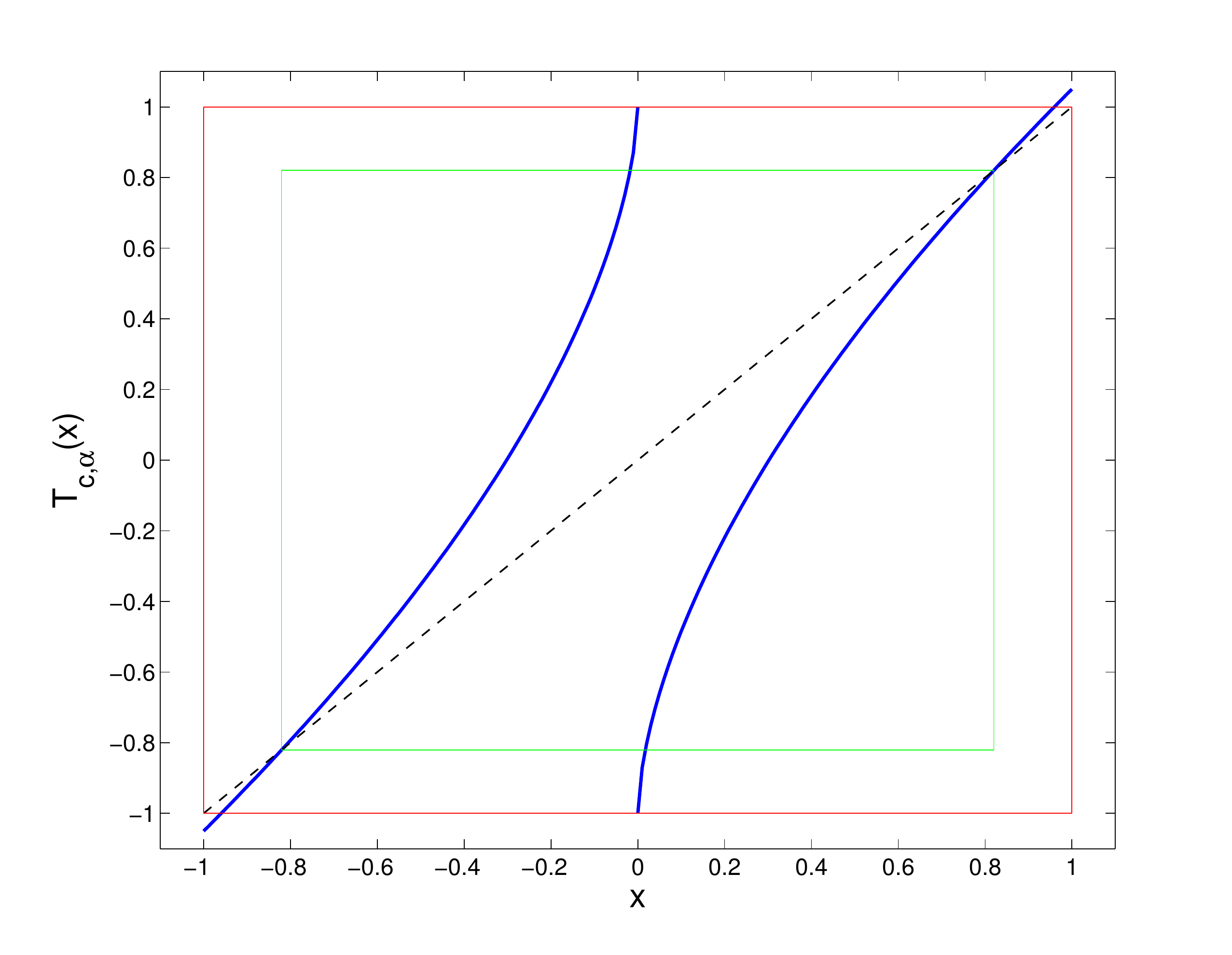}\\
  \caption{Lorenz map~$T_{c,\alpha}$, $c=2.05$, $\al=.6$ (blue). Note that
$[-1,1]\subsetneq T_{c,\alpha}[-1,1]$; bounds of the interval $[-1,1]$ are 
depicted in red (the branches of $T$ extend beyond the red box). The chaotic
repelling set is confined to the interval between two fixed points (green).}
\label{fig:lorenz_map}
  \end{figure}

This family of maps has been studied in connection with the famous Lorenz equations,
\begin{eqnarray}\label{eq:LorenzEq}
\dot{x} &=& \sig (y- x) \nonumber \\  
\dot{y} &=&r x -y -x z \\  
\dot{z} &=&-bz+xy. \nonumber
\end{eqnarray}
We take a relatively standard point of view~\cite{KaplanYorke, Sparrow, GuckenheimerHolmes}, regarding
 $\sig=10$ and $b=8/3$ as fixed, and $r$ as a  parameter. The chaotic 
attractor discovered by Lorenz~\cite{Lorenz63} at $r=28$ has since 
been proved to exist by Tucker~\cite{Tucker} (via computer-assisted methods).
Its formation is now well understood: A 
{\em homoclinic explosion\/} occurs at
$r_{hom}\approx 13.9265$, giving rise to a {\em chaotic saddle\/}. 
As $r$ increases through $r_{het}\approx  24.0579$, heteroclinic 
connections between $(0,0,0)$ and a symmetric pair of periodic 
orbits~$\Gamma^\pm$ appear and the chaotic saddle becomes an 
attractor~$\Omega$. The orbits $\Gamma^\pm$ disappear in subcritical 
Hopf bifurcations at $r_{Hopf}\approx 24.7368$ (parameter values 
from~\cite{DoedelKrauskopfOsinga06}). For $r<r_{het}$ almost all orbits 
are asymptotic to one of two fixed points; for $r_{het}<r<r_{Hopf}$ 
orbits may approach one of these fixed points, or the attractor~$\Omega$; 
for $r>r_{Hopf}$ almost all orbits are attracted to $\Omega$.

Maps like \eqref{eq:LorenzMap} model this situation
via the following reductions. First, solutions to the ODEs~\eqref{eq:LorenzEq} induce a flow on $\R^3$; from this, a {\em return map\/} to the section $\Sigma=\{(x,y,z)~:~z=r-1\}$ may be constructed. This two-dimensional map is an open dynamical system, since not all orbits return to $\Sigma$\footnote{For example, the stable manifold to the fixed point $(0,0,0)$ intersects~$\Sigma$, and some orbits of the flow travel directly to $(0,0,0)$ after leaving~$\Sigma$.}. For `pre-turbulent' $r\in(r_{hom},r_{het})$, the chaotic
saddle admits a {\em strong stable foliation\/}; the return map to $\Sigma$ may be
further reduced by identifying points on the same stable leaf, resulting in one-dimensional
models such as~\eqref{eq:LorenzMap}. The discontinuity at $x=0$ corresponds to the intersection of the stable manifold of $(0,0,0)$ with $\Sigma$; the 
exponent $0<\al<1$ 
is derived from the eigenvalues of the linearization of the system at the origin, 
$\al=\frac{|\lam_s|}{\lam_u}$.  
The parameter $c$ controls how `open' the map is: when $c\leq 2$, the system is closed, and when $c>2$, the $1$-step 
survivor set $X_0$ has the form $X_0=[-x_{c,\al}, x_{c,\al}]$, where 
$x_{c,\al}=(2/c)^{1/\al}$; this is illustrated in red in Figure~\ref{fig:lorenz_map}.

The escape rates of  the system $T_{c,\al}$ for parameters $0<\al<1$, $2<c<3$ 
are illustrated in Figure~\ref{fig:lorenz_leading_eval}.
Figure~\ref{fig:lorenz1d_quasi-conformal}~(left)
illustrates the cumulative distribution functions of the 
quasi-conformal measures, $\mu_{c,\al}$, for $c=2.01$ and 
various values of $\al$. The densities of the accims with respect to 
Lebesgue are illustrated in Figure~\ref{fig:lorenz1d_variousalpha} for 
several $\al$ values. For  $\al<0.5$, the densities become concentrated 
near the endpoints, as the $\al=0.45$ plot in 
Figure~\ref{fig:lorenz1d_variousalpha}~(right) illustrates.

\begin{figure}
\centering
  \includegraphics[width=7cm]{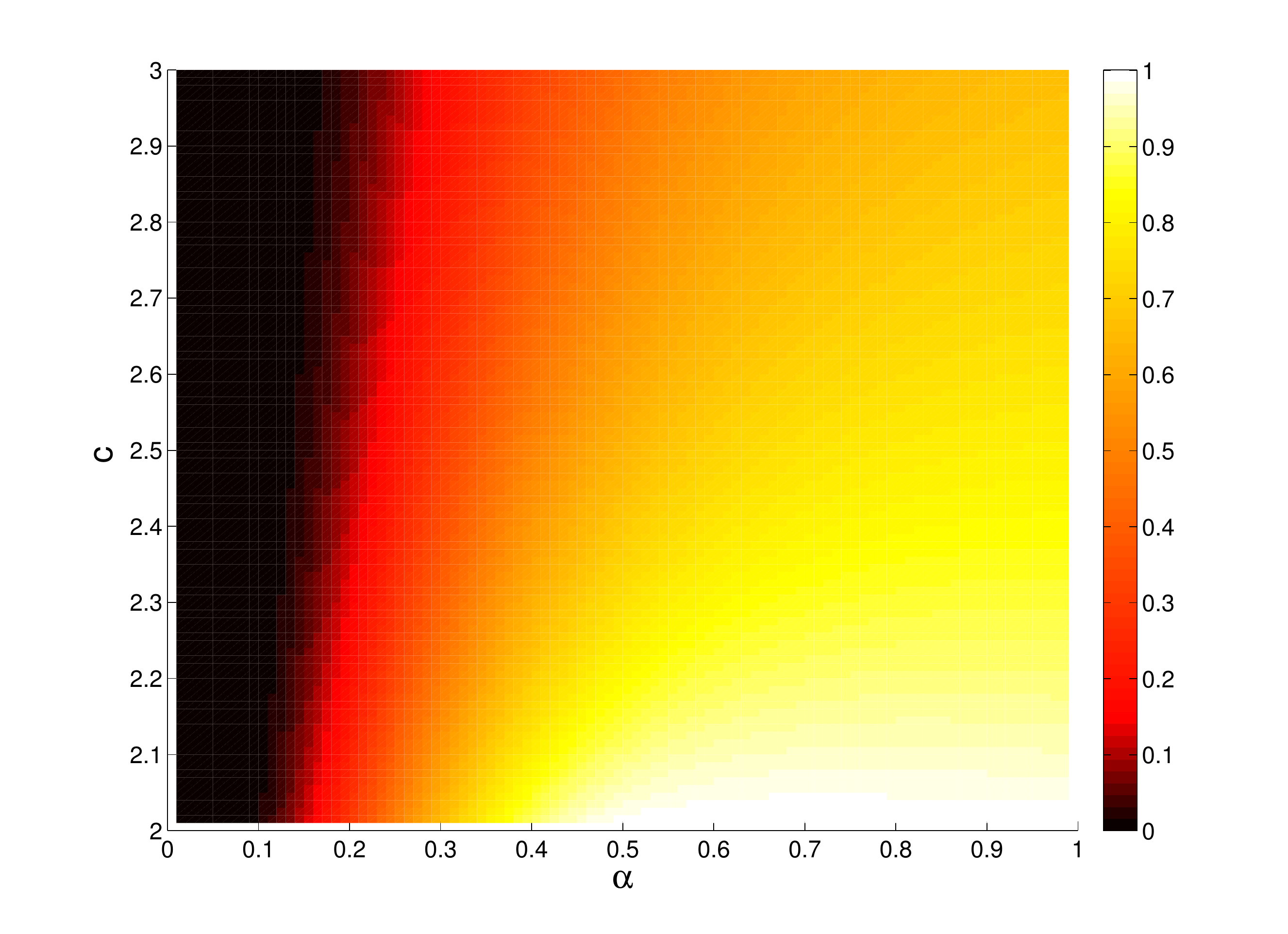}\qquad
 \includegraphics[width=7cm]{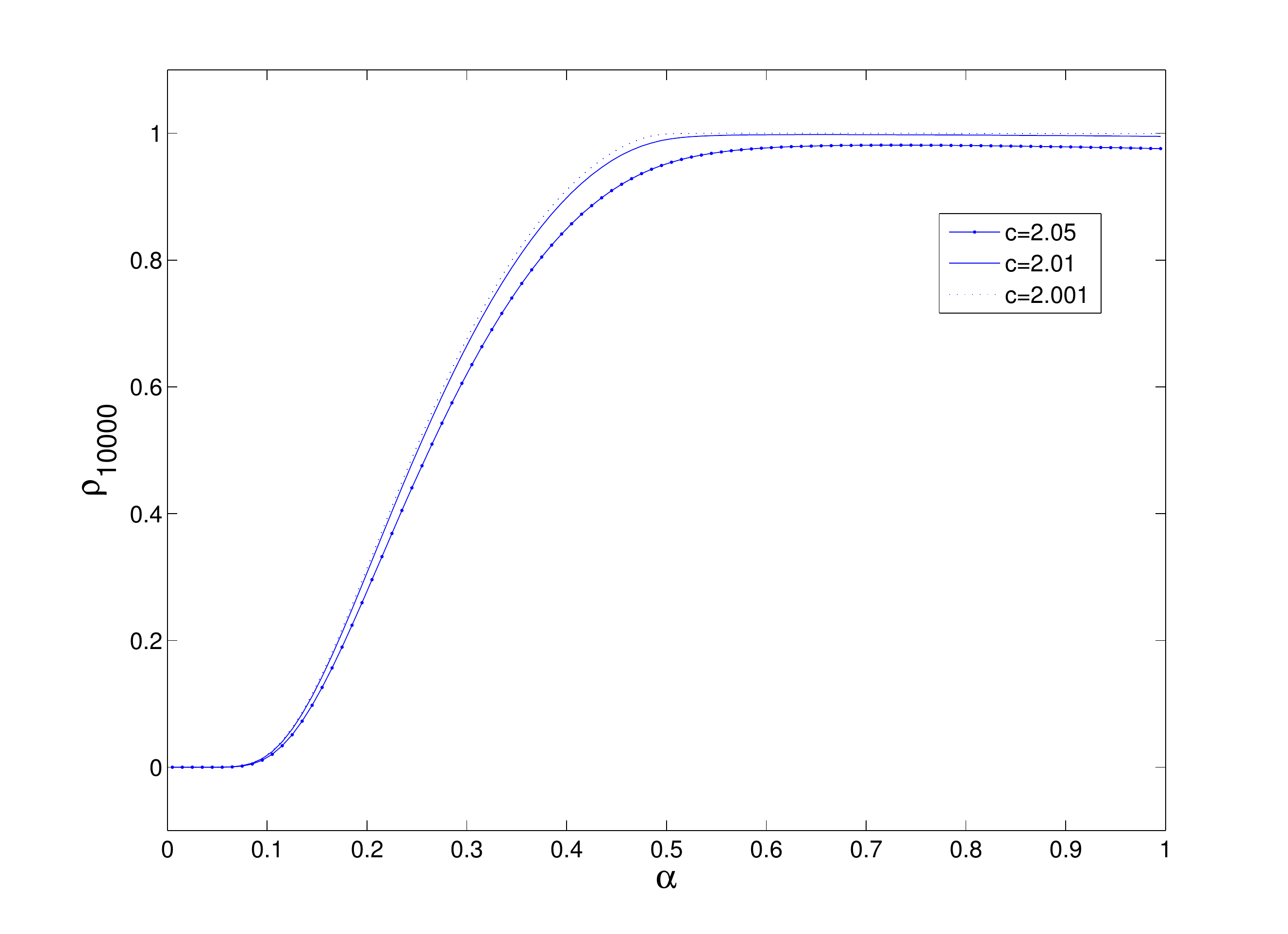}\\
  \caption{Numerical estimate of escape rates via open Ulam method with $k=10 000$ bins. Left: coloured image of leading eigenvalue $\rho_{10000}$ for a range of $\alpha$ and $c$ (light for $\rho$ near 1, dark near $0$). Right:  $\rho_{10000}$ as a function of $\alpha$ for $c=2.05,2.01,2.001$.}\label{fig:lorenz_leading_eval}
  \end{figure}

The escape rate results for these one-dimensional maps can be interpreted
coherently with respect to the behaviour of the Lorenz 
system~\eqref{eq:LorenzEq} (although the scenarios differ according 
to whether $\al\lessgtr 1/2$).
\begin{itemize}
\item Regarding $T_{c,\al}$ as a map on $\R$, for each value of $\alpha\in(0,1)$ and
$c>2$ there are two pairs of fixed points: repellors at $\pm y_{c,\alpha}\in(-1,1)$
(illustrated in green in Figure~\ref{fig:lorenz_map})
and an attracting {\em outer\/} pair $\pm z_{c,\al}$ with $|z_{c,\al}|>1$
(beyond the domain of Figure~\ref{fig:lorenz_map}). 
The {\em inner points\/} $\pm y_{c,\al}$ correspond  to the periodic orbits~$\Gamma^\pm$ from the Lorenz flow, and the outer pair correspond to the
attracting fixed points of the flow. 
\item At some $c=c_*(\al)\leq 2$ the inner and outer pairs coalesce in a saddle-node bifurcation and for $c<c_*$ the only attractor is a chaotic absolutely continuous invariant measure supported on $[-1,1]$.
\item[($\al>1/2$)] Each $T_{c,\al}$ is uniformly expanding on $X_0$ for $c>2$. For
$c>2$ there is a chaotic repellor in $X_0$, and a fully supported accim on $[-1,1]$.
Lebesgue a.e. orbit escapes and is asymptotic to one of the ``outer fixed points''. At
$c=2$ the points $\pm x_{c,\al}=\pm1=\pm y_{c,\al}$ become fixed points, 
with $T'(\pm 1)=2\,\al>1$. The open system thus `closes up' as $c$ decreases to $2$;
this corresponds to the bifurcation point $r_{het}$ in the Lorenz flow (where the origin
connects to $\Gamma^\pm$). For values of $c<2$, $T_{c,\al}$ admits an acim 
(which can be accessed numerically by Ulam's method) and the quasiconformal measure is simply Lebesgue measure. The approach of $\rho_k$ to $1$ as $c\to 2$ can be seen in Figure~\ref{fig:lorenz_leading_eval}, and the close agreement of $\mu_{2.01,\al}$ with Lebesgue measure can be seen in 
Figure~\ref{fig:lorenz1d_quasi-conformal}~(left) for $\al=0.95$. 
\item[($\al<1/2$)] For $c>2$, $T_{c,\al}$ is open on $[-1,1]$, but the uniform expansion
property fails for $c$ sufficiently close to $2$. Indeed, when $c=2$ the fixed points at $\pm1$ are the {\em outer pair\/} $\pm z_{c,\al}$ and $T'(\pm1) <1$. 
For $c\in (c_*,2)$, these attractors $\pm z_{c,\al}\in [-1,1]$ and {\em coexist with a chaotic repellor in $[-y_{c,\al},y_{c,\al}]$\/}.
Fortunately, for $c>c_*$ the open system $T_{c,\al}$ with hole 
$I\setminus [-y_{c,\al}, y_{c,\al}]$ is a Lasota-Yorke map with holes, because it is piecewise expanding\footnote{Corollary~\ref{cor:checkQCFullBranch} shows that admissibility of the open system is implied if 
\begin{equation}\label{eq:qcLorenz}
|T_{c,\al}'(y_{c,a})|^{-1}=\sup_{x\in [-y_{c,\al}, y_{c,\al}]} |T_{c,\al}'(x)|^{-1}< \inf \mcl_{c,\al}1(x).
\end{equation}
This condition can be verified directly.}.
For $c>2$, our main theorem holds for the application of Ulam's method to $T_{c,\al}$ on 
$[-y_{c,\al},y_{c,\al}]$. However, it is simple to extend this result to $[-1,1]$: all points in the intervals $\pm(y_{c,\al},1)$ escape in finitely many iterations,
and corresponding cells of the partitions used in Ulam's method are ``transient''. 
The leading eigenvalue from Ulam's method and approximate quasi-conformal measure on 
$[-y_{c,\al},y_{c,\al}]$ agree with those computed on $[-1,1]$. The approximate accims agree (modulo scaling) between $\pm y_{c,\al}$, the only difference is that the different
$X_0$s lead to a different concentration of mass on preimages of the hole. The approximated escape rates are displayed in Figure~\ref{fig:lorenz_leading_eval},
and concentration of accim on the hole (neighbourhoods of $\pm1$) is evident in Figure~\ref{fig:lorenz1d_variousalpha}~(right). Note also that 
Figure~\ref{fig:lorenz1d_quasi-conformal}~(left) depicts some 
approximate quasiconformal measures for $c=2.01$ and $\alpha<0.45$.
\end{itemize}

\begin{figure}
\centering
  \includegraphics[width=7cm]{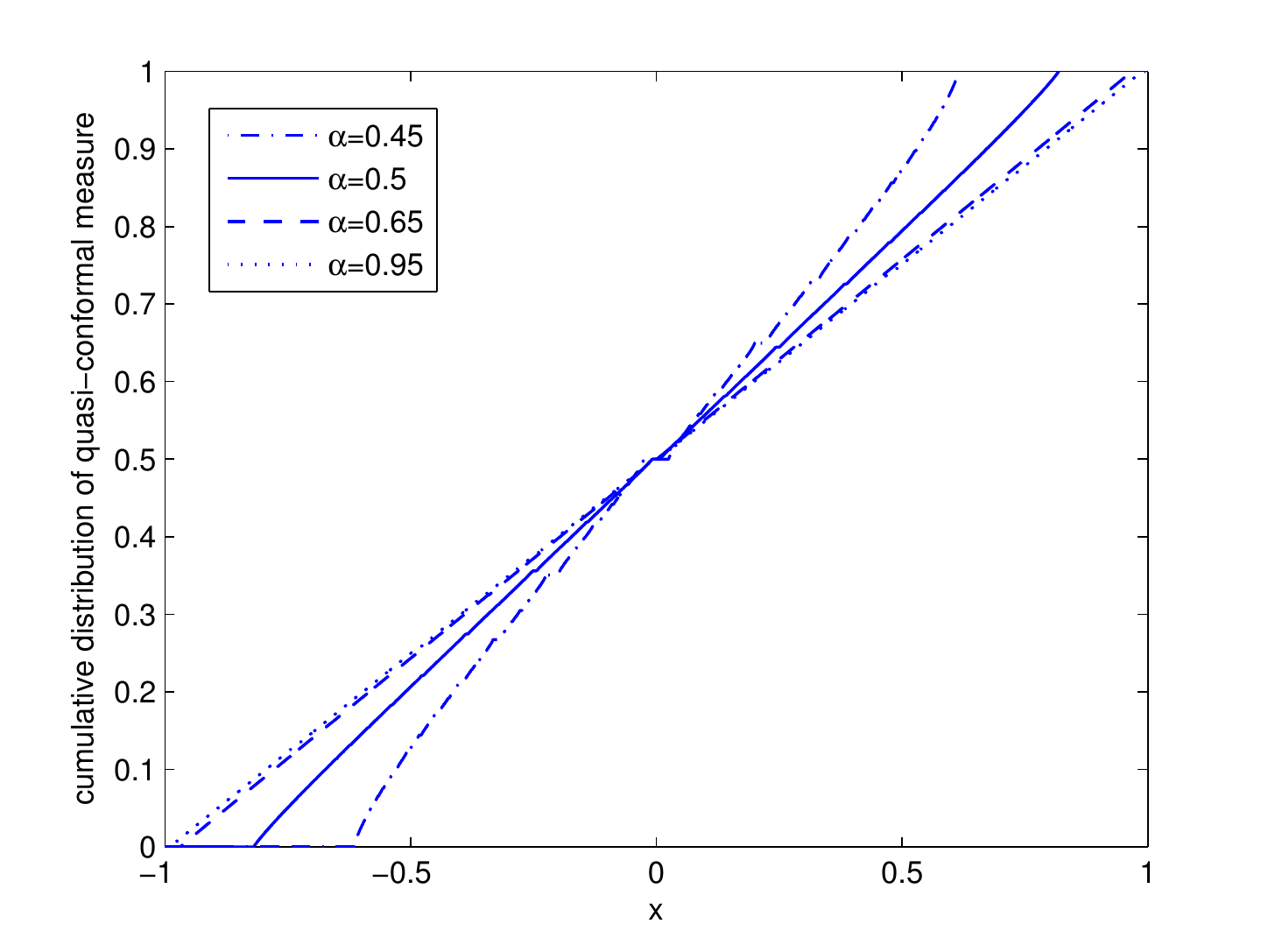}\qquad
  \includegraphics[width=7cm]{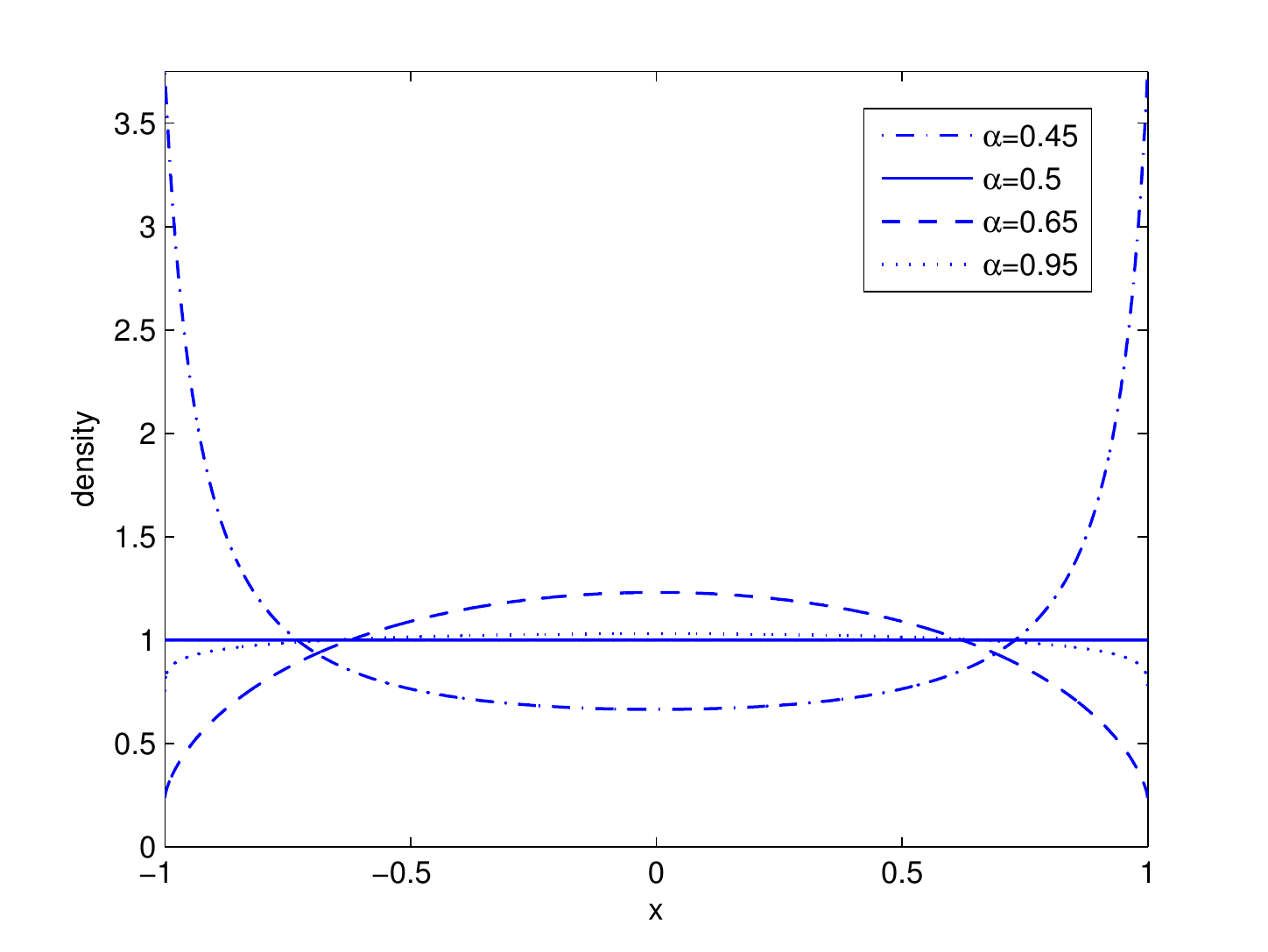}\\

  \caption{Open Ulam approximations for $T_{c,\al}$ ($k=20000$).
Left: cumulative distribution functions for $\mu_{2.01, \al}$ 
where $\al=0.45,0.5, 0.65, 0.95$. Right: accims for $T_{2.01,\al}$ 
(same $\al$).}
\label{fig:lorenz1d_quasi-conformal}
\label{fig:lorenz1d_variousalpha}
  \end{figure}

\section{Proofs}\label{S:pfs}

\subsection{Auxiliary lemmas}\label{subS:AuxLem}
Under the assumptions of Theorem~\ref{thm:LiveMaume}, the quasi-conformal measure $\mu$  of $(\hat{T}, H_0)$ satisfies some further properties that will be exploited in our approach. The measure $\mu$ can be used to define a useful cone of functions in $BV$. For each $a>0$ let
$$\mc{C}_a=\{0\leq f\in BV~:~\var(f)\leq a \mu(f)\}.$$

Combining the result of Lemmas~4.2 and 4.3 from~\cite{LiveMaume} with the argument in the proof of Lemma~3.7 (therein), the conditions 
on~$T$ imply the existence of  a constant $a_1>0$ such that for any $a>a_1$ there is an $\epsilon_a>0$ and $N\in\mathbb{N}$  such that
\begin{equation}\label{eqn:LYcone}\ 
\mcl^N\mc{C}_a \subseteq \mc{C}_{a-\epsilon_a}.
\end{equation}
The values of $N$, $a_1$ and $\epsilon_a$ are all computable in terms of the constants associated with $T$.
We present a modified version of these arguments, based on the classical work of Rychlik \cite{Rychlik}, that specialize to the case $N=1$, and allow us to improve some of the constants involved in the estimates of \cite{LiveMaume}.

\begin{lem}\label{lem:LY1step}
Let $(\hat{T}, H_0)$ be a Lasota-Yorke map with an $\ep$-Ulam-admissible hole.
Then, there exists $K_\ep>0$ such that for every $f\in BV$,
\[
  \var(\mcl f) \leq \al_\ep \var(f)+K_\ep \mu(|f|).
\]
Furthermore, there is a constant $a_1>0$ such that for any $a>a_1$ there is an $\epsilon_a>0$ such that
\begin{equation}\label{eqn:LYcone1step}\ 
\mcl\mc{C}_a \subseteq \mc{C}_{a-\epsilon_a}.
\end{equation}
\end{lem}

\begin{proof}
Let $\mc{Z}$ be the monotonicity partition for $\hat{T}$. Define $\hat{g}:I \to \R$ by
$\hat{g}(x)=|D\hat{T}(x)|^{-1}$ for every $x\in \Big(I\setminus \bigcup_{Z\in \mc{Z}} \partial Z \Big) \cup \{0,1\}$, and $\hat{g}(x)=0$ otherwise.
We obtain the following Lasota-Yorke inequality by adapting 
the approach of Rychlik  \cite[Lemmas 4-6]{Rychlik}. Let $\mathcal{Z}_\epsilon  \in \mathcal{G}_\epsilon$. 
Then,
\[
  \var (\hat{\mcl} f)\leq \var (f \hat{g}) \leq 
  (2+\ep)\|D\hat{T}^{-1}\|_{\infty}\var(f)+ \|D\hat{T}^{-1}\|_{\infty}(1+\ep) \sum_{A\in \mc{Z}_\ep} \inf_{A} |f|.
\]
We slightly modify $\hat{g}$ to account for the jumps at the hole $H_0$, and define
$g:I \to \R$ by $g=1_{X_0}\hat{g}$. 
Now, only elements of $\mathcal{Z}_\ep^*$ contribute to the variation of                                                                                   
$\hat{\mc{L}}f$, and we get
\begin{align*}
\var (\mcl f) &= \var (\hat{\mcl} (1_{X_0} f)) \leq \var (f (1_{X_0}\hat{g}))=\sum_{A\in \mc{Z}_\ep^*}\var_{A} (f (1_{X_0}\hat{g}))\\
&\leq \sum_{A\in \mc{Z}_\ep^*}\var_{A} (f) \|1_{X_0}\hat{g}\|_{\infty}+  \|1_A f\|_{\infty} \var_A(1_{X_0}\hat{g})\\
&\leq \sum_{A\in \mc{Z}_\ep^*}\var_{A} (f) \|DT^{-1}\|_{\infty}+  \Big(\inf_A |f| +\var_A(f)\Big) \var_A(g).
\end{align*}
Thus, since for every $A \in \mc{Z}_\ep^*$, $\var_A(g) \leq \|DT^{-1}\|_{\infty}(1+\ep)$, one has that
\begin{equation}\label{eq:LYRychlik}
\var (\mcl f)
\leq (2+\ep)\|DT^{-1}\|_{\infty} \var(f)+ \sum_{A\in \mc{Z}_\ep^*} \|DT^{-1}\|_{\infty}(1+\ep) \inf_A |f|.
\end{equation}
Now we proceed as in the proof of \cite[Lemma 2.5]{LiveMaume}, and observe that there exists $\del>0$ such that if $A\in \mc{Z}_{\ep,g}$, then
\begin{equation}\label{eq:infMu}
  \inf_A |f| \leq \del^{-1}\mu(1_A|f|),
\end{equation}
whereas if $A\in \mc{Z}_{\ep,b}$, we let  $A'\in \mc{Z}_{\ep,g}$ be the nearest good partition element\footnote{It is shown in \cite[Lemma 2.4]{LiveMaume} that whenever $(T,H_0)$ is an open system with an admissible hole, then $\mc{Z}_{\ep,g}\neq \emptyset$}, and get
\[
  \inf_A |f| \leq \inf_{A'} |f| + \var_{I(A,A')}(f),
\]
where $I(A,A')$ is an interval that contains $A$ and has as an endpoint $x_{A'}\in A'$, fixed in advance, such that, after possibly redefining $f$ at the discontinuity points of $f$, $|f(x_{A'})|=\inf_{A'}|f|$. 
Notice that either $I(A,A')\subseteq I_-(A')$ or $I(A,A')\subseteq I_+(A')$, where
$ I_+(A')$ is the union of $A'_+:=A' \cap\{x: x\geq x_{A'}\} $ with the contiguous elements of $\mc{Z}_{\ep,b}$ on the right of $A'$, and $ I_-(A')$ is defined in a similar manner.
Thus,
\begin{equation}\label{eq:infBad}
  \sum_{A\in \mc{Z}_{\ep,b}}\inf_A |f| \leq \xi_\ep \var(f) + 2\xi_\ep  \sum_{A'\in \mc{Z}_{\ep,g}}\inf_{A'} |f|,
\end{equation}
where the factor 2 appears due to the fact that a single good interval could have at most $\xi_\ep$ bad intervals on the left and $\xi_\ep$ bad intervals on the right.  Combining equations \eqref{eq:infMu} and \eqref{eq:infBad}, we get
\[
  \sum_{A\in \mc{Z}_{\ep}^*} \inf_A |f| \leq  \xi_\ep\var(f)+\del^{-1}(1+2\xi_\ep) \sum_{A'\in \mc{Z}_{\ep,g}} \mu(1_{A'}|f|).
\]
Plugging back into \eqref{eq:LYRychlik}, we get
\begin{align*}
\var (\mcl f)
&\leq \|DT^{-1}\|_{\infty}(2+\ep+\xi_\ep) \var(f)+ \|DT^{-1}\|_{\infty}(1+\ep)\del^{-1}(1+2\xi_\ep)\mu(|f|).
\end{align*}
We get the first part of the lemma by choosing $K_\ep=\|DT^{-1}\|_{\infty}(1+\ep)\del^{-1}(1+2\xi_\ep)$.
For the second part, we recall that $\mu(\mcl f)=\rho \mu(f)$, so for every $f\in \mc{C}_a$, we have that
\begin{align*}
  \frac{\var(\mcl f)}{\mu(\mcl f)} \leq \frac{\al_\ep}{\rho}a+ \frac{K_\ep}{\rho}.
\end{align*}
Thus, $\mcl f \in \mc{C}_a$, provided $a>\frac{K_\ep}{\rho-\al_\ep}=:a_1$.
\end{proof}

Moving toward a $BV, L^1(\leb)$ Lasota-Yorke inequality, we have the following.
\begin{lem}\label{lem:muvsLeb}
 Let $\zeta>0$ be given. Then there is a constant $B_\zeta<\infty$ such that
$$\mu(f) \leq B_\zeta\,|f|_1 + \zeta\,\var(f),$$
for $0\leq f\in BV(I)$.
\end{lem}

\begin{proof} Let $\mathcal{Z}^{(n)}$ be the $n$-fold  monotonicity partition for $T_0$
where $n$ is such that $\mu(Z)<\frac{\zeta}{2}$ for all $Z\in\mathcal{Z}^{(n)}$.
This choice is possible in view of \cite[Lemma~3.10]{LiveMaume}.
Choose $k$ such that every $\frac{1}{k}$ subinterval intersects at most two such $Z$.
Then, if $K$ is a subinterval of length~$1/k$, there are elements $Z_1,Z_2\in\mathcal{Z}^{(n)}$ such that $K\subset Z_1\cup Z_2$; hence $\mu(K)<\zeta = \zeta\,k\,m(K)$. Now let $\xi$ be a partition of $I$ into
subintervals of length~$1/k$ and put
$$F = \sum_{K\in\xi} \esssup_Kf\,\mathbf{1}_K.$$
Then, $f\leq F$ and $F-f \leq \sum_{K\in\xi}V_K(f)\,\mathbf{1}_K$, where $V_K(f)$ denotes the variation of $f$ inside the interval $K$. Thus,
$$|F-f|_1 \leq \sum_{K\in\xi}V_K(f)\,m(K)=V_I(f)/k.$$
We now estimate
\begin{align*}
\int f\,d\mu \leq \int F\,d\mu&=\sum_{K\in\xi} \esssup_Kf \mu(K)\\
&\leq \sum_{K\in\xi} \esssup_Kf \zeta \\
&=\zeta\,k\,|F|_1\\
&=\zeta\,k\,|f|_1 + \zeta\,k\,|F-f|_1\\
&\leq \zeta\,k\,|f|_1 + \zeta\,V_I(f).
\end{align*}
Putting $B_\zeta=\zeta\,k$ completes the proof.
\end{proof}

A direct consequence of Lemmas~\ref{lem:LY1step} and \ref{lem:muvsLeb} is the following.
\begin{cor}\label{cor:LYineq}
Let $\alpha_\ep<\alpha<\rho$,  where $\al_\ep$ is defined in Equation~\eqref{defnAl_ep}. Then, there exists $K>0$ such that
\begin{equation}\label{eq:LY}
\var(\mathcal{L}f) \leq \alpha\,\var(f) + K |f|_1.
\footnote{For convenience, we have dropped the $\ep$ dependence on $\al$ and $K$. This should cause no confusion in the sequel, as $\ep$ is fixed throughout the section.}
\end{equation}
\end{cor}
\begin{proof}
Let $\zeta' = \frac{\alpha-\alpha_\ep}{K_\ep}$, where $K_\ep$ comes from Lemma~\ref{lem:LY1step}.
Let $B_{\zeta'}$ be given by Lemma~\ref{lem:muvsLeb}.
Then, Lemma~\ref{lem:LY1step} ensures
\begin{align*}
\var(\mathcal{L}f) &\leq \alpha_\ep\,\var(f) + K_\ep(B_{\zeta'} |f|_1 + \zeta'\,\var(f))\\
&=\alpha\,\var(f) + K |f|_1.
\end{align*}
\end{proof}

Another useful result regarding the relation between the Ulam approximations and the accim and quasi-conformal measure is the following.
\begin{lem}\label{lem:Garys} 
There exists $n>0$ such that $(P_k^n)_{ij}>0$ for all $i,j$                                                                                                                      
satisfying $\mu(I_i)>0$ and $\int_{I_j} h\ dm>0$.                                                                                                                                   
\end{lem}
                                                                                                                                                                                
\begin{proof}
Fix $i,j$ satisfying the hypotheses. By Theorem~\ref{thm:LiveMaume}, $\|(\mathcal{L}^n\chi_i)/\rho^n-\mu(I_i)h\|_\infty\to 0$ as $n\to                                                                                                            
\infty$.                                                                                                                                                                            
Choose $n_{ij}$ large enough so that $\int_{I_j}                                                                                                                                    
\mathcal{L}^{N}\chi_i\ dm>0$ for all $N\geq n_{ij}$.                                                                                                                                                  
Because there are a finite number of $I_i$ and $I_j$ we can put                                                                                                                     
$n=\max_{i,j}n_{ij}$ and obtain $\int_{I_j} \mathcal{L}^{n}\chi_i\ dm>0$                                                                                                            
for all $i,j$ satisfying the hypotheses.                                                                                                                                            
Note that this implies $\int_{I_j} (\pi_k\mathcal{L})^n\chi_i\ dm>0$                                                                                                                
because the support of the integrand is possibly enlarged by taking Ulam projections.                                                                                                                          
This now implies $(P^n_k)_{ij}>0$. 
\end{proof}

\subsection{Proof of the main result}\label{sec:pfMainThm}
The lemmas presented in \S\ref{subS:AuxLem} allow us to derive parts (I) and (II) of Theorem~\ref{MainThm} via the perturbative approach from \cite{KellerLiverani98}.
Indeed, Theorem~\ref{thm:LiveMaume} shows that $\rho>\al$ is the leading eigenvalue of $\mcl$, and that it is simple.
Furthermore, $\mcl_k$ is a small perturbation of $\mcl$ for large $k$,
in the sense that $\sup_{\|f\|_{BV}=1} |(\mcl_k-\mcl) f|_1\to 0$ as $k\to \infty$.
Indeed, 
\begin{align*}
\sup_{\|f\|_{BV}=1} |(\mcl_k-\mcl) f|_{1} &=\sup_{\|f\|_{BV}=1} |(\pi_k-Id)\mcl f|_1\leq 
\sup_{ \|f\|_{BV}=\|\mcl\|_{BV} } |(\pi_k-Id)f|_1 \\
&\leq \|\mcl\|_{BV} \max_{I_j\in \mc{P}_k} m(I_j),
\end{align*}
and the latter is proportional to $\tau_k$, the diameter of the partition, which tends to 0 as $k\to \infty$.

Since $\pi_k$ decreases variation, Corollary~\ref{cor:LYineq} implies the uniform inequality
\begin{equation}\label{eq:LYUlam}
\var(\mathcal{L}_k f) \leq \alpha\,\var(f) + K |f|_1, \quad \forall k\in \N,
\end{equation}
which is the last hypothesis to check to be in the position to apply the perturbative machinery of \cite{KellerLiverani98}. 
This result ensures that for sufficiently large $k$, $\mcl_k$ has a simple eigenvalue $\rho_k$ near $\rho$, and its corresponding eigenvector $h_k\in BV$ converges to $h$ in $L^1(\leb)$, giving the convergence statements in \eqref{it:I} and \eqref{it:II}.

In order to show \eqref{it:III}, we consider the operator $\bar{\mcl}_k:=\mcl_k \circ \pi_k$.
In view of Lemma~\ref{lem:QCMeasureFromEvector}, $\bar{\mcl}_k^*\mu_k=\rho_k \mu_k$, and $ \bar{\mcl}_k h_k=\rho_k h_k$.
As in the previous paragraph, one can check that $\bar{\mcl}_k$ is a small perturbation of $\mcl$.
In fact, 
\[
\sup_{\|f\|_{BV}=1} |(\bar{\mcl}_k-\mcl) f|_1\leq 2\max_{I_j\in \mc{P}_k} m(I_j)=2\tau_k.
\]
Also, the Lasota-Yorke inequality \eqref{eq:LY} holds with $\mcl$ replaced by $\bar{\mcl}_k$.
Thus, \cite[Corollary~1]{KellerLiverani98} (see (iii) below) shows that for large $k$, $\rho_k$ is the leading eigenvalue of $\bar{\mcl}_k$.

Let $\Pi_k$ be the spectral projectors defined by
$$\Pi_k:= \frac{1}{2\pi\,i}\oint_{\partial B_\delta(\rho)}(z-\bar{\mcl}_k)^{-1}\,dz,$$
where $\delta$ is small enough to exclude all spectrum of $\mcl$ apart from the peripheral eigenvalue $\rho$. 
Also let $$\Pi_0:= \frac{1}{2\pi\,i}\oint_{\partial B_\delta(\rho)}(z-\mcl)^{-1}\,dz.$$
Then,~\cite[Corollary~1]{KellerLiverani98} provides $K_1,K_2>0$, and $\eta\in(0,1)$ for which
\begin{enumerate}[(i)]
\item $|(\Pi_k-\Pi_0)f|_1 \leq K_1\,{\tau_k}^\eta\,\|f\|_{BV}$,
\item \label{it:ii} $\|\Pi_k f\|_{BV} \leq K_2\,|\Pi_k f|_1$,
\item For large enough $k$, $\text{rank}(\Pi_k)=\text{rank}(\Pi_0)$.
\end{enumerate}
Since $\rho$ is simple and isolated, this setup implies that for large enough~$k$, each $\Pi_k$ is a bounded, rank-$1$ operator on $BV$:
$$\Pi_k = \mu_k(\cdot)\,h_k,$$
where each $h_k\in BV$, $\bar{\mcl}_k h_k = \rho_k\,h_k$ and $\rho_k\in B_\delta(\rho)$.
Since $h_k=\Pi_k h_k$ we can choose $|h_k|_1=1$ so that $\|h_k\|_{BV}\in [1,K_2]$.
Now, let $g\in BV$. Then, by the above,
\begin{align*}
|\mu_k(g)-\mu(g)| & = |(\mu_k(g)-\mu(g))h_k|_1\leq |\mu_k(g)h_k-\mu(g)\,h|_1 + | \mu(g) (h_k-h)|_1\\
&= |\Pi_{k}(g) - \Pi_0(g)|_1 + |\mu(g)|\,|h_k-h|_1 \to 0, \quad \text{as } k\to \infty.
\end{align*}
Since $\mu$ and $\mu_k$ are in fact measures, the above is enough to show that $\mu_k \to \mu$ in the weak-$*$ topology.

In particular, there is a $k_0$ such that $\mu_k(h)>0$ for all
$k\geq k_0$.
To show the last claim of \eqref{it:III}, we will show  that if $\mu_k(h)>0$ then $\supp(\mu)\subseteq\supp(\mu_k)$. 
Let $\psi_k$ be a leading right eigenvector of $\bar{\mcl}_k$ such that  $P_k\psi_k=\rho_k\psi_k$ and
$[\psi_k]_l=\frac{\mu(I_l)}{m(I_l)}$ ($l=1,\ldots,k$). Choose $i$ such that
$\mu(I_i)>0$, $j$ such that $[\psi_k]_j=\int_{I_j}h\,dm=\int_{I_j}h\,d\mu_k>0$ and $n\geq n_{ij}$ as in Lemma~\ref{lem:Garys}. Then,
$$[\psi_k]_i =\rho^{-n}[{P_k}^n\psi_k]_i\geq \rho^{-n}[{P_k}^n]_{ij}[\psi_k]_j>0.$$
This establishes that $\mu_k(I_i)>0$ and hence that $\supp(\mu)\subseteq \cup\{I_i~:~\mu(I_i)>0\}\subset \supp(\mu_k)$, as claimed.

For the quantitative statement of \eqref{it:I}, note that for every $f\in BV$, $0=(\mcl-\rho I)h = (\mcl-\rho I)\Pi_0 f$, so that
$$(\rho_k-\rho)h_k = (\bar{\mcl}_k-\mcl)h_k + (\mcl-\rho)(\Pi_k-\Pi_0)h_k.$$
Hence,
\begin{align*}
|\rho_k-\rho|\,|h_k|_1 &\leq  2\tau_k \|h_k\|_{BV} + (|\mcl|_1+|\rho|)K_1\,{\tau_k}^\eta\,\|h_k\|_{BV}\\
&\leq  2(\tau_k  + (1+|\rho|)K_1\,{\tau_k}^\eta)\,K_2\,|h_k|_1,
\end{align*}
where $K_1, K_2$ and $\eta$ are as above.
This gives the error bound $|\rho_k-\rho|\leq O({\tau_k}^\eta)$.
\qed

\subsection{Proof of Lemma~\ref{lem:EnlargingHoles}}\label{pflem:EnlargingHoles}

Let $\mcl_m$ be the transfer operator associated to $T_m$. That is, $\mcl_m(f)=\hat{\mcl}(1_{X_{m}}f)$. Then, $\mcl_m^n(f)=\hat{\mcl}^n(1_{X_{m+n-1}}f)$, and therefore,
\begin{equation}\label{eq:relTransferOps}
\hat{\mcl}^m\circ \mcl_m^n= \mcl_0^{m+n}.
\end{equation}
Hence, an interval is good for $T_0$ if and only if it is good for $T_m$ for every $m$.
In the rest of this proof we will say an interval is good if it is good for either (and therefore all) $T_m$.

Let $\mc{Z}_0=\mc{Z}\vee \mc{H}_0$, where $\mc{H}_0$ is the partition of $H_0$ into intervals, and we recall that $\mc{Z}$ is the monotonicity partition of $\hat{T}$.
Let $\mc{G}_\ep$ be an $\ep$-adequate partition for $T_0$. Then, a partition $\mc{G}_{\ep,m}$ may be constructed by cutting each element of $\mc{G}_\ep \vee \mc{Z}_0^{(m)}$ in at most $K$ pieces, where $K$ is independent of $m$, in such a way that the variation requirement $\max_{Z\in \mc{G}_{\ep,m}} \var_Z(\hat{g}1_{X_m})\leq \|DT_m^{-1}\|_\infty(1+\ep)$ is satisfied, and thus $\mc{G}_{\ep,m}$ is an $\ep$-adequate partition for $T_m$.  
Indeed, $K=2+\big \lceil \|\hat{g}\|_{\infty}/ \essinf(\hat{g}) \big\rceil$ is a possible choice.
The term 2 allows one to account for possible jumps at the boundary                                           
points of $H_m$, as there are at most two of them in each $Z\in \mc{G}_\ep \vee \mc{Z}_0^{(m)}$. The term                                                                                     
$M=\lceil \|\hat{g}\|_\infty/\essinf(\hat{g}) \rceil$                                                         
allows one to split each interval $Z\in \mc{G}_\ep \vee \mc{Z}_0^{(m)}$                                         
into at most $M$ subintervals $Z_1, \dots, Z_M$, in such a way that                                           
for every $1\leq j \leq M$, $\var_{int(Z_j)}(\hat{g}1_{X_m})\leq                                               
(1+\ep)\|\hat{g}1_{X_m}\|_{\infty}.$                                                                          
The chosen value of $M$ is necessary to account for the possible                                              
discrepancy between $\|\hat{g}1_{X_0}\|_{\infty}$ and                                                         
$\|\hat{g}1_{X_m}\|_{\infty}$. (Recall also that $\hat{g}$ is continuous                                      
on each $int(Z_j)$.)      

Now, let $b=\#\mc{Z}_0$. Then, each bad interval of $\mc{G}_\ep$ gives rise to at most $Kb^m$ (necessarily bad) intervals in $\mc{G}_{\ep,m}$. When a good interval of  $\mc{G}_\ep$ is split, it also gives rise to at most $Kb^m$ intervals in $\mc{G}_{\ep,m}$. In this case some of the intervals may be bad, but it is guaranteed that at least one of them remains good, as being good is equivalent to having non-zero $\mu$ measure.
Thus, the number of contiguous bad intervals in $\mc{G}_{\ep,m}$ is at most $Kb^m(B+2)$, where $B$ is the number of contiguous bad intervals in $\mc{G}_{\ep}$. Therefore, $\t{\xi}_\ep(T_m)=\exp\Big( \limsup_{n\to \infty} \frac1n \log (1+\xi_{\ep,n}(T_m)) \Big)\leq \t{\xi}_\ep(T_0)$.

Clearly, $\t{\Theta}(T_m)\leq \t{\Theta}(T_0)$. Finally, we will show that $\rho(T_0)\leq \rho(T_m)$.
Recall that $\rho_j$ is the leading eigenvalue of $\mcl_j$. Let $f\in BV$ be nonzero and such that $\mcl_0 f=\rho_0 f$. We claim that $\mcl_m (1_{X_{m-1}}f)=\rho_0 1_{X_{m-1}} f$, which yields the inequality, because necessarily $1_{X_{m-1}}f\neq 0$. Indeed,
\begin{align*}
  \rho_0 1_{X_{m-1}} f &=1_{X_{m-1}} \mcl_0 f 
=1_{X_{m-1}} \mcl_m f = \mcl_m f = \mcl_m( 1_{X_{m-1}} f),
\end{align*}
where the second equality follows from the fact that $\mcl_0 (1_{H_m} f)$ is supported on $T(H_m)=H_{m-1}$. The third one, from the fact that $\mcl_m f$ is supported on $T(X_m)\subseteq X_{m-1}$. The last one, because $\mcl_m(1_{H_{m-1}}f)=0$.

The first statement of the lemma follows.
The relations between escape rates, accims and quasi-conformal measures follow from comparing via Equation~\eqref{eq:relTransferOps} the statements of part (4) of Theorem~\ref{thm:LiveMaume} applied to $T_0$ and $T_m$. 
\qed

\subsection{Proof of Lemma~\ref{lem:AdmissibleAndUlamAdmissible}}\label{subS:AdmVsUlamAdm}\hfill \\
Assume $H_0$ is an  $\ep$-admissible hole for $\hat{T}$.
Then, $T^n:=(\hat{T}^n, H_{n-1})$ is an open Lasota-Yorke map.
Fix $\t{\Theta}<\eta<\rho$ so that for all $n$ sufficiently large,
\[
 \exp( \frac1n \log \|  (DT^n)^{-1} \|_{\infty})\exp( \frac1n \log (1+\xi_{\ep,n}))<\eta.
 \]
Then, $\|  (DT^n)^{-1} \|_{\infty}\xi_{\ep,n}<\eta^n$. By possibly making $n$ larger, we can assume that 
$(2+\ep)\|  (DT^n)^{-1} \|_{\infty}<\eta^n$, and that $2\eta^n <\rho^n$.
Then,
$\|  (DT^n)^{-1} \|_{\infty}(2+\ep+\xi_{\ep,n})<\rho^n$.

We remark that $\xi_\ep(T^n)=\xi_{\ep,n}(T)$.
Thus $\al_\ep(T^n)=\|  (DT^n)^{-1} \|_{\infty}(2+\ep+\xi_{\ep,n})$.
Furthermore, in view of Theorem~\ref{thm:LiveMaume},
$\rho(T^n)=\lim_{m\to\infty} \inf_{x\in D_{mn}} \frac{\mcl^{n(m+1)}1(x)}{\mcl^{nm}1(x)}=\mu(\mcl^{n}1)=\rho^n$.
\qed

\subsubsection*{Acknowledgments}
The authors thank Banff International Research Station (BIRS), where the present work was started, for the splendid working conditions provided.
CB's work is supported by an NSERC grant.
GF is partially supported by the UNSW School of Mathematics and an ARC Discovery Project
(DP110100068), and thanks the Department of Mathematics and Statistics at the University of Victoria for hospitality. 
CGT was partially supported by the Pacific Institute for the Mathematical Sciences (PIMS)
and NSERC. 
RM thanks the Department of Mathematics and          
Statistics (University of Victoria) for hospitality during part of the period when this paper was           
written.


\end{document}